\documentclass[12pt, reqno]{amsart}
\usepackage{amsmath, amsthm, amscd, amsfonts, amssymb, graphicx, color}
\usepackage[bookmarksnumbered, colorlinks, plainpages]{hyperref}
\hypersetup{colorlinks=true,linkcolor=red, anchorcolor=green, citecolor=cyan, urlcolor=red, filecolor=magenta, pdftoolbar=true}

\newtheorem{theorem}{Theorem}[section]
\newtheorem{lemma}[theorem]{Lemma}

\theoremstyle{definition}
\newtheorem{definition}[theorem]{Definition}

\theoremstyle{remark}
\newtheorem{remark}[theorem]{Remark}
\numberwithin{equation}{section}

\begin{document}

\setcounter{page}{1}

\title[Fujita-type results for the degenerate parabolic equations]{Fujita-type results for the degenerate parabolic equations on the Heisenberg groups}

\author[A. Z. Fino, M. Ruzhansky, B. T. Torebek]{Ahmad Z. Fino, Michael Ruzhansky, Berikbol T. Torebek$^*$}

\address{\textcolor[rgb]{0.00,0.00,0.84}{Ahmad Z. Fino \newline Department of Mathematics, College of Engineering and Technology, \newline American University of the Middle East, Kuwait}}
\email{\textcolor[rgb]{0.00,0.00,0.84}{ahmad.fino01@gmail.com}}
\address{\textcolor[rgb]{0.00,0.00,0.84}{Michael Ruzhansky \newline Department of Mathematics: Analysis, Logic and Discrete Mathematics \newline Ghent University, Belgium \newline
 and \newline School of Mathematical Sciences \newline Queen Mary University of London, United Kingdom}}
\email{\textcolor[rgb]{0.00,0.00,0.84}{michael.ruzhansky@ugent.be}}
\address{\textcolor[rgb]{0.00,0.00,0.84}{Berikbol T. Torebek \newline Department of Mathematics: Analysis, Logic and Discrete Mathematics \newline Ghent University, Belgium \newline and \newline Institute of
Mathematics and Mathematical Modeling \newline 125 Pushkin str.,
050010 Almaty, Kazakhstan}}
\email{\textcolor[rgb]{0.00,0.00,0.84}{berikbol.torebek@ugent.be }}

\thanks{This research has been funded by the Science Committee of the Ministry of Education and Science of the Republic of Kazakhstan (Grant No. AP14869090), by the FWO Odysseus 1 grant G.0H94.18N: Analysis and Partial Differential Equations, and by the Methusalem programme of the Ghent University Special Research Fund (BOF) (Grant number 01M01021). Michael Ruzhansky is also supported by EPSRC grants EP/R003025/2 and EP/V005529/1.}

\let\thefootnote\relax\footnote{$^{*}$Corresponding author}



\begin{abstract} In this paper, we consider the Cauchy problem for the degenerate parabolic equations on the Heisenberg groups with power law non-linearities. We obtain Fujita-type critical exponents, which depend on the homogeneous dimension of the Heisenberg groups. The analysis includes the case of porous medium equations. Our proof approach is based on methods of nonlinear capacity estimates specifically adapted to the nature of the Heisenberg groups. We also use the Kaplan eigenfunctions method in combination with the Hopf-type lemma on the Heisenberg groups.
\end{abstract}
\maketitle
\tableofcontents
\section{Introduction}
The main purpose of this paper is to study the following two types of degenerate parabolic equations on the Heisenberg groups:
$$v_{t}=\Delta_{\mathbb{H}}v^{m}+v^\sigma,\qquad {t>0,\,\,\eta\in \mathbb{H}^n,}$$ and
$$u_{t}=u^{q}\Delta_{\mathbb{H}}u+u^p,\qquad t>0,\,\,\eta\in \mathbb{H}^n,$$
where $n\geq1$, $m\geq1$, $\sigma>1$, $q\geq0$, $p>1$. The Heisenberg group  is the Lie group $\mathbb{H}^n=\mathbb{R}^{2n+1}$ equipped with the following law
$$\eta\circ\eta^\prime=(x+x^\prime,y+y^\prime,\tau+\tau^\prime+2(x\cdotp y^\prime-x^\prime\cdotp y)),$$
where $\eta=(x,y,\tau)$, $\eta^\prime=(x^\prime,y^\prime,\tau^\prime)$, and $\cdotp$ is the scalar product in $\mathbb{R}^n$. The homogeneous Heisenberg norm is defined by
$$|\eta|_{_{\mathbb{H}}}=\left(\left(\sum_{i=1}^n (x_i^2+y_i^2)\right)^2+\tau^2\right)^{\frac{1}{4}}=\left((|x|^2+|y|^2)^2+\tau^2\right)^{\frac{1}{4}},$$
where $|\cdotp|$ is the Euclidean norm associated to $\mathbb{R}^n$. The left-invariant vector fields that span the Lie algebra are given by
$$X_i=\partial_{x_i}-2 y_i\partial_\tau,\qquad Y_i=\partial_{y_i}+2 x_i\partial_\tau.$$
The Heisenberg gradient is given by
\begin{equation}\label{48}
\nabla_{\mathbb{H}}=(X_1,\dots,X_n,Y_1,\dots,Y_n),
\end{equation}
and the sub-Laplacian is defined by
\begin{equation}\label{40}
\Delta_{\mathbb{H}}=\sum_{i=1}^n(X_i^2+Y_i^2)=\Delta_x+\Delta_y+4(|x|^2+|y|^2)\partial_\tau^2+4\sum_{i=1}^{n}\left(x_i\partial_{y_i\tau}^2-y_i\partial_{x_i\tau}^2\right),
\end{equation}
where $\Delta_x=\nabla_x\cdotp\nabla_x$ and $\Delta_y=\nabla_y\cdotp\nabla_y$ stand for the Laplace operators on $\mathbb{R}^n$. The homogeneous dimension of $\mathbb{H}^n$ is $Q=2n+2.$

We will obtain the results about nonexistence of global nontrivial solutions for various values of exponents $\sigma$ and $p$.
\subsection{Historical background}
\subsubsection{Results on $\mathbb{R}^n$} In \cite{Fuj66}, Fujita studied the following semi-linear heat equation
\begin{equation}\label{Fuj}
\begin{cases}
u_{t}(x,t)-\Delta u(x,t)=u^{p}(x, t), \,\,\, (x,t)\in \mathbb{R}^{n}\times (0,\infty), \\{}\\ u(x,0)=u_{0}(x)\geq0, \,\,\, x\in \mathbb{R}^{n}.
 \end{cases}
\end{equation}
It was shown that, if $1<p<1+\frac{2}{n}$, then problem \eqref{Fuj} admits no nontrivial positive global solutions, while, if $p>1+\frac{2}{n}$, then problem \eqref{Fuj} admits global positive solutions for some sufficiently small initial data. Later, in \cite{Hayakawa} Hayakawa proved that, if $p=1+\frac{2}{n}$, then problem \eqref{Fuj} admits no nontrivial positive global solutions. The number $p_F=1+\frac{2}{n}$ is called the Fujita critical exponent.

In \cite{Gal1}, Galaktionov et al. considered the porous medium equation with power nonlinearity
\begin{equation}\label{PME1}
\begin{cases}
u_{t}(x,t)-\Delta u^m(x,t)=u^{p}(x, t), \,\,\, (x,t)\in \mathbb{R}^{n}\times (0,\infty),\,m>1, \\{}\\ u(x,0)=u_{0}(x)\geq0, \,\,\, x\in \mathbb{R}^{n},
 \end{cases}
\end{equation} and established the following results:\\
(i) let $m<p< m+\frac{2}{n},$ then the solution of \eqref{PME1} does not exist globally in time;\\
(ii) let $p>m+\frac{2}{n},$ then the problem \eqref{PME1} has a global solution for some sufficiently small initial data.

In \cite{Moch1} the authors proved that, if $p=m+\frac{2}{n}$, then problem \eqref{PME1} has no positive global solutions.

When $m > 1,$ by using the transformation $v(x, t) = au^m(bx, t),$ $a = m^{m/(p-1)},$ $b = m^{(p-m)/2(p-1)},$ the porous medium equation \eqref{PME1} can be transformed to the degenerate parabolic equation
\begin{equation}\label{Deg1}
v_{t}-v^k\Delta v=v^r, \,\,\, (x,t)\in \mathbb{R}^{n}\times (0,\infty),
\end{equation} where $0<k=\frac{m-1}{m}<1$ and $r=\frac{m+p-1}{m}>1.$

In \cite{Gal2}, Galaktionov et al. obtained the following results for the equation \eqref{Deg1}:\\
(i) let $1<r<k+1+\frac{2}{n(1-k)},$ then the solution of \eqref{Deg1} does not exist globally in time;\\
(ii) let $r>k+1+\frac{2}{n(1-k)},$ then there are both global solutions and solutions blowing up in finite time.

In \cite{Wink1} Winkler extended the results of \cite{Gal2} by taking the more general $k\geq 1$ in \eqref{Deg1}. In particulary, Winkler obtained the following results:\\
(i) For $1 \leq r < k + 1$ (resp. $1 \leq r < \frac{3}{2}$ if $k= 1$), all positive solutions of \eqref{Deg1} are global but unbounded, provided that $u_0$ decreases sufficiently fast in space;\\
(ii) For $r = k + 1,$ all positive solutions of \eqref{Deg1} blow up in finite time;\\
(iii) For $r > k + 1,$ there are both global and non-global positive solutions, depending on the size of $u_0.$

It follows from the above results that the equation \eqref{Deg1} has two type of critical exponents
$$r_c=\left\{\begin{array}{ll}
k+1+\frac{2}{n(1-k)}&\,\,\,\text{for}\,\,\,0<k<1,\\\\
k+1&\,\,\,\text{for}\,\,\,k\geq 1.\\
\end{array}
\right.
$$
\subsubsection{Sub-elliptic extensions} In \cite{Zhang} Zhang considered the semilinear diffusion equation on the Heisenberg groups:
\begin{equation}\label{DifHeis}
u_t-\Delta_{\mathbb{H}} u=|u|^{p},\,\,\,\,\,{t>0,\,\,\eta\in \mathbb{H}^n,}
\end{equation} and they proved that, if $1<p< 1+\frac{2}{Q}$, $Q=2n+2,$ then the problem \eqref{DifHeis} admits no positive global solutions.
Later, Pohozhaev and V\'{e}ron \cite{PohVer} studied a more general parabolic equation on $\mathbb{H}^n,$ and proved that there is no global solutions for $1<p\leq 1+\frac{2}{Q}$.
In \cite{Georgiev}, Georgiev and Palmieri proved sharp lifespan estimates for local in time solutions of problem \eqref{DifHeis}.

In \cite{Pascucci, Pasc2} Pascucci obtained the Fujita-type results for the semilinear diffusion equation on Carnot groups. We also note that the nonexistence of global solutions to the various semilinear parabolic equations on the Heisenberg group were studied by many authors (see for example \cite{Samet, D'Ambrosio, Han, Kirane1, Yang}).

Recently, the second author and Yessirkegenov \cite{Ruzhansky} consider the following equation on general unimodular Lie groups $\mathbb{G},$
\begin{equation}\label{DifLie}
u_t-\Delta_{\mathbb{G}} u=|u|^{p},\,\,\,\,\,\text{in}\,\,\,\, (0,\infty)\times\mathbb{G},\end{equation} and established the following results:

\emph{Let $\mathbb{G}$ be a connected unimodular Lie group with polynomial volume growth of order $D$ and let $1<p<\infty:$\\
(i) let $1<p\leq 1+\frac{2}{D},$ then \eqref{PME1} does not admit any nontrivial global solution;\\
(ii) let $p>1+\frac{2}{D},$ then \eqref{PME1} has a global solution for some small initial data.}

\emph{Let $\mathbb{G}$ be a connected unimodular Lie group with exponential volume growth, then, for any $1<p<\infty$ the equation \eqref{PME1} has a global solution for some positive initial data.}

\emph{Let $\mathbb{G}$ be a compact Lie group, then, for any $1<p<\infty$ the equation \eqref{PME1} does not admit any nontrivial nonnegative solutions.}

\subsubsection{Motivation} From the above reasoning, it is easy to see that Fujita-type results for semilinear parabolic equations on manifolds are fairly well studied. However, such results have been little studied for strongly nonlinear parabolic equations. Here we can note some papers devoted to the study of the Fujita-type results for the p-Laplacian diffusion equations and porous medium equations on Riemannian manifolds \cite{Chen2, Grillo, Punzo1, Punzo2}.

As far as we know, in the case of sub-Riemannian manifolds, there are only a couple of papers \cite{Samet, Kirane1}, where Fujita-type results for the $p$-Laplacian diffusion equations on Heisenberg groups are obtained. We also note that in \cite{Ruzhansky1} the blow-up results were obtained for the porous medium equations on a bounded domain of the Carnot groups.

Motivated by this fact, in this paper we consider two types of strongly non-linear parabolic equations on the Heisenberg groups. In particular, we determine the critical exponents for which the considered equations are globally unsolvable.

\section{Porous medium equation}

In this section we consider the following porous medium equation
\begin{equation}\label{2}
\left\{
\begin{array}{ll}
\,\,\displaystyle{v_{t}=\Delta_{\mathbb{H}}v^{m}+v^\sigma,}&\displaystyle {t>0,\,\,\eta\in \mathbb{H}^n,}\\
\\
\displaystyle{v(0,\eta)= v_0(\eta)}\geq 0,&\displaystyle{\eta\in \mathbb{H}^n, }\\
\\
\displaystyle{v(t,\eta)\geq0},&\displaystyle {t>0,\,\,\eta\in \mathbb{H}^n,}
\end{array}
\right. \end{equation}
where $v_0\in L^1_{loc}(\mathbb{H}^n)$, $n\geq1$, $m\geq1$, $\sigma>1$.\\
\begin{definition}\textup{(Weak solution of \eqref{2})}${}$\\
Let $0\leq v_0\in L^1_{loc}(\mathbb{H}^n)$ and $T>0$. We say that $v$ is a nonnegative weak solution of \eqref{2} on $[0,T)\times\mathbb{H}^n$ if
$$v\in L_{loc}^\sigma((0,T)\times\mathbb{H}^n)\cap L_{loc}^m((0,T)\times\mathbb{H}^n)\cap L_{loc}^\infty((0,T);L_{loc}^1(\mathbb{H}^n)),$$
and
\begin{equation}\label{weaksolution2}\begin{split}
&\int_{\mathbb{H}^n}v(\tau,\eta)\psi(\tau,\eta)\,d\eta-\int_{\mathbb{H}^n}v(0,\eta)\psi(0,\eta)\,d\eta\nonumber\\
&=\int_0^\tau\int_{\mathbb{H}^n}v^\sigma\psi(t,\eta)\,d\eta\,dt+\int_0^\tau\int_{\mathbb{H}^n}v^{m}\,\Delta_{\mathbb{H}}\psi(t,\eta)\,d\eta\,dt+\int_0^\tau\int_{\mathbb{H}^n}v\psi_t(t,\eta)\,d\eta\,dt
\end{split}\end{equation}
holds for all compactly supported $\psi\in C^{1,2}_{t,x}([0,T)\times\mathbb{H}^n)$, and $0\leq\tau<T$. If $T=\infty$, we call $v$ a global  in time weak solution to \eqref{2}.
\end{definition}

\begin{theorem}\label{theo1}
Let $m>1,$ $\sigma>1$ and let $v_0(\eta)\geq0,\,v_0(\eta)\not\equiv0, \eta\in\mathbb{H}^n$. \\
\begin{itemize}
\item[(i)] Suppose that $v_0\in L^1(\mathbb{H}^n).$ If $$m<\sigma\leq m+\frac{2}{Q},$$
then problem \eqref{2} has no nonnegative global weak solutions.\\
\item[(ii)] Assume that $v_0\in L^1(\mathbb{H}^n)$, and there exists a constant $\varepsilon>0$ such that, for every $0<\gamma<Q$, the initial datum verifies the following assumption:
$$v_0(\eta)\geq \varepsilon (1+|\eta|_{_{\mathbb{H}}}^2)^{-\gamma/2}.$$
If
$$m<\sigma<m+\frac{2}{\gamma},$$
then problem \eqref{2} has no nonnegative global weak solutions.\\
\item[(iii)] Assume that $v_0\in L^1(\mathbb{H}^n)\cap L^\infty(\mathbb{H}^n)$. If $\sigma=m,$ then problem \eqref{2} has no nonnegative global weak solutions $v\in  L_{loc}^\infty((0,\infty);L^\infty(\mathbb{H}^n))$.\end{itemize}
\end{theorem}
\begin{remark} When $m=1$, the critical exponent $\sigma_c=m+\frac{2}{Q}$ coincides with the critical exponent obtained in \cite{Zhang} for the semilinear diffusion equations on $\mathbb{H}^n.$

Observe that in the case $\gamma \geq Q,$ we have $$m+\frac{2}{\gamma}\leq m+\frac{2}{Q}.$$ Consequently, part (ii) of Theorem \ref{theo1} follows immediately from the part (i).
\end{remark}
\begin{remark} Note that in Theorem \ref{theo1} there are no results about the global existence of a solution. We expect that, for $\sigma>m+\frac{2}{Q}$ and for sufficiently small initial data there should exist a global solution. Due to technical difficulties, we left this question open.
\end{remark}
To prove Theorem \ref{theo1}, below we give a number of auxiliary results.
\begin{lemma}\label{lemma2}
Let $u,v$ be twice differentiable real valued functions defined on $\mathbb{H}^n$. Then
$$\Delta_{\mathbb{H}}(uv)=\Delta_{\mathbb{H}}(u)v+2\nabla_{\mathbb{H}}(u)\nabla_{\mathbb{H}}(v)+u\Delta_{\mathbb{H}}(v).$$
\end{lemma}
\begin{proof}
It is easy to see that
\begin{equation}\label{41}
\Delta_x(uv)=\Delta_{x}(u)v+2\nabla_{x}(u)\nabla_{x}(v)+u\Delta_{x}(v),
\end{equation}
\begin{equation}\label{42}
\Delta_y(uv)=\Delta_{y}(u)v+2\nabla_{y}(u)\nabla_{y}(v)+u\Delta_{y}(v),
\end{equation}
and
\begin{equation}\label{43}
\partial^2_\tau(uv)=\partial^2_\tau(u)v+2\partial_\tau(u)\partial_\tau(v)+u\partial^2_\tau(v).
\end{equation}
Moreover,
\begin{equation}\label{44}
\begin{split}\sum_{i=1}^{n}x_i\partial_{y_i\tau}^2(uv)&=v\sum_{i=1}^{n}x_i\partial_{y_i\tau}^2(u)+\partial_\tau (u)\sum_{i=1}^{n}x_i\partial_{y_i}(v)\\&+\partial_\tau (v)\sum_{i=1}^{n}x_i\partial_{y_i}(u)+u\sum_{i=1}^{n}x_i\partial_{y_i\tau}^2(v),
\end{split}\end{equation}
and
\begin{equation}\label{45}
\begin{split}\sum_{i=1}^{n}y_i\partial_{x_i\tau}^2(uv)&=v\sum_{i=1}^{n}y_i\partial_{x_i\tau}^2(u)+\partial_\tau (u)\sum_{i=1}^{n}y_i\partial_{x_i}(v)\\&+\partial_\tau (v)\sum_{i=1}^{n}y_i\partial_{x_i}(u)+u\sum_{i=1}^{n}y_i\partial_{x_i\tau}^2(v).
\end{split}\end{equation}
Using \eqref{41}-\eqref{45} and the definition of $\Delta_{\mathbb{H}}$, we obtain
\begin{align*}
\Delta_{\mathbb{H}}(uv)&=\Delta_{\mathbb{H}}(u)v+u\Delta_{\mathbb{H}}(v)\\
&\quad+2\left[\nabla_{x}(u)\nabla_{x}(v)+\nabla_{y}(u)\nabla_{y}(v)+4(|x|^2+|y|^2)\partial_\tau(u)\partial_\tau(v)\right.\\
&\quad+2\partial_\tau (v)\sum_{i=1}^{n}x_i\partial_{y_i}(u)-2\partial_\tau (v)\sum_{i=1}^{n}y_i\partial_{x_i}(u)\\
&\quad\left.+2\partial_\tau (u)\sum_{i=1}^{n}x_i\partial_{y_i}(v)-2\partial_\tau (u)\sum_{i=1}^{n}y_i\partial_{x_i}(v)\right]\\
&=\Delta_{\mathbb{H}}(u)v+u\Delta_{\mathbb{H}}(v)+2\nabla_{\mathbb{H}}(u)\nabla_{\mathbb{H}}(v),
\end{align*} proving the claim.
\end{proof}

\begin{lemma}\label{lemma4}
For $\varepsilon>0$, $A>0$, let
$$\Theta(\eta)=e^{-\varepsilon\left[A+(|x|^2+|y|^2)^2+\tau^2\right]^{\frac{1}{2}}},\qquad \eta=(x,y,\tau)\in \mathbb{H}^n.$$
Then
$$
-\Delta_{\mathbb{H}}\Theta(\eta)\leq 2\varepsilon(Q+2)\Theta(\eta),\qquad\hbox{for all}\,\,\eta\in\mathbb{H}^n.
$$
\end{lemma}
\begin{proof} Let $\rho(x,y,\tau):=A+(|x|^2+|y|^2)^2+\tau^2$. We have
$$\nabla_x \Theta(\eta)=-\frac{\varepsilon}{2}\rho^{-\frac{1}{2}}\nabla_x(\rho)\Theta(\eta),$$
and then
$$
\Delta_x \Theta(\eta)=\frac{\varepsilon}{4}\rho^{-\frac{3}{2}}|\nabla_x(\rho)|^2\Theta(\eta)-\frac{\varepsilon}{2}\rho^{-\frac{1}{2}}\Delta_x(\rho)\Theta(\eta)+\frac{\varepsilon^2}{4}\rho^{-1}|\nabla_x(\rho)|^2\Theta(\eta).
$$
As $\Delta_x(\rho)=(4n+8)|x|^2+4n|y|^2$, we conclude that
\begin{equation}\label{35}
\begin{split}\Delta_x \Theta(\eta)&=\frac{1}{4}\left(\varepsilon\rho^{-\frac{3}{2}}+\varepsilon^2\rho^{-1}\right)|\nabla_x(\rho)|^2\Theta(\eta)\\&-\varepsilon\left((2n+4)|x|^2+2n|y|^2\right)\rho^{-\frac{1}{2}}\Theta(\eta).
\end{split}\end{equation}
Similarly,
\begin{equation}\label{36}
\begin{split}\Delta_y \Theta(\eta)&=\frac{1}{4}\left(\varepsilon\rho^{-\frac{3}{2}}+\varepsilon^2\rho^{-1}\right)|\nabla_y(\rho)|^2\Theta(\eta)\\&-\varepsilon\left(2n|x|^2+(2n+4)|y|^2\right)\rho^{-\frac{1}{2}}\Theta(\eta).
\end{split}\end{equation}
On the other hand,
$$\partial_\tau \Theta(\eta)=-\frac{\varepsilon}{2}\rho^{-\frac{1}{2}}\partial_\tau(\rho)\Theta(\eta),$$
and then
\begin{eqnarray}\label{37}
\partial^2_\tau \Theta(\eta)&=&\frac{\varepsilon}{4}\rho^{-\frac{3}{2}}(2\tau)^2\Theta(\eta)-\frac{\varepsilon}{2}\rho^{-\frac{1}{2}}2\Theta(\eta)+\frac{\varepsilon^2}{4}\rho^{-1}4\Theta(\eta)\nonumber\\
&=&\tau^2\left(\varepsilon\rho^{-\frac{3}{2}}+\varepsilon^2\rho^{-1}\right)\Theta(\eta)-\varepsilon\rho^{-\frac{1}{2}}\Theta(\eta).
\end{eqnarray}
Next,
\begin{align*}
\partial^2_{y_i\tau} \Theta(\eta)&=-\varepsilon\tau\partial_{y_j}\left(\rho^{-\frac{1}{2}}\Theta(\eta)\right)\\
&=-\varepsilon\tau\left(-\frac{1}{2}\rho^{-\frac{3}{2}}\partial_{y_j}(\rho)\Theta(\eta)-\frac{\varepsilon}{2}\rho^{-1}\partial_{y_j}(\rho)\Theta(\eta)\right)\\
&=\frac{\tau}{2}\left(\varepsilon\rho^{-\frac{3}{2}}+\varepsilon^2\rho^{-1}\right)\partial_{y_j}(\rho)\Theta(\eta)\\
&=2\tau\left(\varepsilon\rho^{-\frac{3}{2}}+\varepsilon^2\rho^{-1}\right)y_i(|x|^2+|y|^2)\Theta(\eta),
\end{align*}
for all $1\leq i\leq n$, which implies that
\begin{equation}\label{38}
\sum_{i=1}^nx_i\partial^2_{y_i\tau} \Theta(\eta)=2\tau\left(\varepsilon\rho^{-\frac{3}{2}}+\varepsilon^2\rho^{-1}\right)x\cdotp y(|x|^2+|y|^2)\Theta(\eta).
\end{equation}
Similarly,
\begin{equation}\label{39}
\sum_{i=1}^ny_i\partial^2_{x_i\tau} \Theta(\eta)=2\tau\left(\varepsilon\rho^{-\frac{3}{2}}+\varepsilon^2\rho^{-1}\right)x\cdotp y(|x|^2+|y|^2)\Theta(\eta).
\end{equation}
Using \eqref{36}-\eqref{39} in \eqref{40}, we arrive at
\begin{align*}
\Delta_{\mathbb{H}}\Theta(\eta)&=\left[\frac{|\nabla_x(\rho)|^2+|\nabla_y(\rho)|^2}{4}+\tau^2(|x|^2+|y|^2)\right]\left(\varepsilon\rho^{-\frac{3}{2}}+\varepsilon^2\rho^{-1}\right)\Theta(\eta)\\& -4\varepsilon(n+2)(|x|^2+|y|^2)\rho^{-\frac{1}{2}}\Theta(\eta)\\
&\geq-4\varepsilon(n+2)(|x|^2+|y|^2)\rho^{-\frac{1}{2}}\Theta(\eta),
\end{align*}
and finally, using the fact that $(|x|^2+|y|^2)^2\leq \rho\Longrightarrow(|x|^2+|y|^2)\leq \rho^{\frac{1}{2}}$, we get
$$\Delta_{\mathbb{H}}\Theta(\eta)\geq-4\varepsilon(n+2)\Theta(\eta)=-2\varepsilon(Q+2)\Theta(\eta),$$ completing the proof.
\end{proof}

\begin{lemma}\label{lemma3}  Let $v_0\in L^\infty(\mathbb{H}^n)$ and $T>0$. Let $\psi_0\in C^{1,2}_{t,x}([0,T)\times\mathbb{H}^n)$ be such that
\begin{equation}\label{49}
\int_{\mathbb{H}^n}\{|\psi_0(t,\eta)|+|\partial_t\psi_0(t,\eta)|+|\nabla_{\mathbb{H}}\psi_0(t,\eta)|+|\Delta_{\mathbb{H}}\psi_0(t,\eta)|\}\,d\eta<\infty,\,\forall t\in[0,T).
\end{equation}
 If $0\leq v\in  L_{loc}^\infty((0,T);L^\infty(\mathbb{H}^n))$ is a weak solution of \eqref{2} on $[0,T)\times\mathbb{H}^n$ then
\begin{equation}\label{weaksolution3}\begin{split}
&\int_{\mathbb{H}^n}v(\tau,\eta)\psi_0(\tau,\eta)\,d\eta-\int_{\mathbb{H}^n}v(0,\eta)\psi_0(0,\eta)\,d\eta\nonumber\\
&=\int_0^\tau\int_{\mathbb{H}^n}v^\sigma\psi_0(t,\eta)\,d\eta\,dt +\int_0^\tau\int_{\mathbb{H}^n}v^{m}\,\Delta_{\mathbb{H}}\psi_0(t,\eta)\,d\eta\,dt +\int_0^\tau\int_{\mathbb{H}^n}v\partial_t\psi_0(t,\eta)\,d\eta\,dt
\end{split}\end{equation}
for all $\tau\in[0,T)$.
\end{lemma}
\begin{proof}
Let $0<T\leq\infty$. Suppose that $0\leq v\in  L_{loc}^\infty((0,T);L^\infty(\mathbb{H}^n))$ is a weak solution of (\ref{2}) on $[0,T)\times\mathbb{H}^n$, then we have
\begin{eqnarray*}
&{}&\int_{\mathbb{H}^n}v(\tau,\eta)\psi(\tau,\eta)\,d\eta-\int_{\mathbb{H}^n}v(0,\eta)\psi(0,\eta)\,d\eta\\
&{}&\qquad\qquad=\int_0^\tau\int_{\mathbb{H}^n}v^\sigma\psi(t,\eta)\,d\eta\,dt+\int_0^\tau\int_{\mathbb{H}^n}v^{m}\,\Delta_{\mathbb{H}}\psi(t,\eta)\,d\eta\,dt\\&{}&\qquad\qquad +\int_0^\tau\int_{\mathbb{H}^n}v\psi_t(t,\eta)\,d\eta\,dt,
\end{eqnarray*}
for any compactly supported $\psi\in C^{1,2}_{t,x}([0,T)\times\mathbb{H}^n)$, and $0\leq\tau<T$. Let $\tau\in[0,T)$ be a fixed number, and let
$$\psi(t,\eta):= \varphi_R(\eta)\psi_0(t,\eta):=\varphi_1(x)\varphi_1(y)\varphi_2(\tau) \psi_0(t,\eta),\quad t\in[0,T),\,\,\eta\in\mathbb{H}^n,$$
 with
$$\varphi_1(x):=\Phi\left(\frac{|x|}{R}\right),\quad\varphi_1(y):=\Phi\left(\frac{|y|}{R}\right),\quad\varphi_2(\tau):=\Phi\left(\frac{|\tau|}{R^2}\right),$$
where $R\gg 1$, and $\Phi$ is a smooth nonnegative non-increasing function such that
$$\Phi(r)=\left\{\begin{array}{l}
1\quad\text{if }0\leq r\leq 1/2,\\{}\\
{\searrow}\quad\text{if }1/2\leq r\leq 1,\\{}\\
{0}\quad\text{if } r\geq 1.
\end{array}\right.$$
Then
\begin{eqnarray*}
&{}&\int_{\mathcal{B}}v(\tau,\eta)\varphi_R(\eta)\psi_0(\tau,\eta)\,d\eta-\int_{\mathcal{B}}v(0,\eta)\varphi_R(\eta)\psi_0(0,\eta)\,d\eta\\
&{}&\qquad\qquad=\int_0^\tau\int_{\mathcal{B}}v^\sigma\varphi_R(\eta)\psi_0(t,\eta)\,d\eta\,dt+\int_0^\tau\int_{\mathcal{B}}v^{m}\,\Delta_{\mathbb{H}}(\varphi_R(\eta)\psi_0(t,\eta))\,d\eta\,dt \\&{}&\qquad\qquad+\int_0^\tau\int_{\mathcal{B}}v\varphi_R(\eta)\partial_t\psi_0(t,\eta)\,d\eta\,dt,
\end{eqnarray*}
where
$$\mathcal{B}=\{\eta=(x,y,\tau)\in\mathbb{H}^n;\,\,|x|^2,|y|^2,|\tau|\leq R^2\},$$ and we also denote $$\mathcal{C}=\{\eta=(x,y,\tau)\in\mathbb{H}^n;\,\,\frac{R^2}{2}\leq |x|^2,|y|^2,|\tau|\leq R^2\}.$$
Using Lemma \ref{lemma2}, we get
\begin{eqnarray}\label{50}
&{}&\int_{\mathcal{B}}v(\tau,\eta)\varphi_R(\eta)\psi_0(\tau,\eta)\,d\eta-\int_{\mathcal{B}}v(0,\eta)\varphi_R(\eta)\psi_0(0,\eta)\,d\eta\nonumber\\
&{}&=\int_0^\tau\int_{\mathcal{B}}v^\sigma\varphi_R(\eta)\psi_0(t,\eta)\,d\eta\,dt+\int_0^\tau\int_{\mathcal{C}}v^{m}\,\psi_0(t,\eta)\Delta_{\mathbb{H}}\varphi_R(\eta)\,d\eta\,dt\nonumber\\
&{}&\,\,\,+\,2\int_0^\tau\int_{\mathcal{C}}v^{m}\,\nabla_{\mathbb{H}}(\varphi_R(\eta))\nabla_{\mathbb{H}}(\psi_0(t,\eta))\,d\eta\,dt\nonumber\\
&{}&\,\,\,+\int_0^\tau\int_{\mathcal{B}}v^{m}\,\varphi_R(\eta)\Delta_{\mathbb{H}}\psi_0(t,\eta)\,d\eta\,dt+\int_0^\tau\int_{\mathcal{B}}v\varphi_R(\eta)\partial_t\psi_0(t,\eta)\,d\eta\,dt.
\end{eqnarray}
On the other hand,
$$I_R:=\left|\int_0^\tau\int_{\mathcal{C}}v^{m}\,\psi_0(t,\eta)\Delta_{\mathbb{H}}\varphi_R(\eta)\,d\eta\,dt\right|\leq\int_0^\tau\int_{\mathcal{C}}v^{m}\,|\psi_0(t,\eta)| \left|\Delta_{\mathbb{H}}\varphi_R(\eta)\right|\,d\eta\,dt.$$
Using \eqref{40}, we have
\begin{eqnarray*}
\left|\Delta_{\mathbb{H}}\varphi_R(\eta)\right|&=&\left|\Delta_{\mathbb{H}}\left(\varphi_1(x)\varphi_1(y)\varphi_2(\tau) \right)\right|\\
&\leq&\left|\Delta_x\varphi_1(x)\right|\varphi_1(y)\varphi_2(\tau)+\,\varphi_1(x)\left|\Delta_y\varphi_1(y)\right|\varphi_2(\tau)\\
&{}&+\,4(|x|^2+|y|^2)\varphi_1(x)\varphi_1(y)\left|\partial_\tau^2\varphi_2(\tau)\right| \\&{}&+4\sum_{j=1}^{n}|x_j|\varphi_1(x)\left|\partial_{y_j}\varphi_1(y)\right|\left|\partial_\tau\varphi_2(\tau)\right| \\
&{}&+4\sum_{j=1}^{n}|y_j|\varphi_1(y)\left|\partial_{x_j}\varphi_1(x)\right|\left|\partial_\tau\varphi_2(\tau)\right|,
\end{eqnarray*}
on $\mathcal{C}$. Substituting $\varphi_1$ and $\varphi_2$  we get
\begin{eqnarray*}
\left|\Delta_{\mathbb{H}}\varphi_R(\eta)\right|&\leq&\left|\Delta_x\left(\Phi\left(\frac{|x|}{R}\right)\right)\right|\Phi\left(\frac{|y|}{R}\right) \Phi\left(\frac{|\tau|}{R^2}\right)+\Phi\left(\frac{|x|}{R}\right)\left|\Delta_y\left(\Phi\left(\frac{|y|}{R}\right)\right)\right|\Phi\left(\frac{|\tau|}{R^2}\right)\\
&{}&+\,4(|x|^2+|y|^2)\Phi\left(\frac{|x|}{R}\right)\Phi\left(\frac{|y|}{R}\right)\left|\partial^2_\tau\left(\Phi\left(\frac{|\tau|}{R^2}\right)\right)\right| \\
&{}&+\,4\sum_{j=1}^{n}|x_j|\Phi\left(\frac{|x|}{R}\right)\left|\partial_{y_j}\left(\Phi\left(\frac{|y|}{R}\right)\right)\right|\left|\partial_\tau\left(\Phi\left(\frac{|\tau|}{R^2}\right)\right)\right| \\
&{}&+\,4\sum_{j=1}^{n}|y_j|\Phi\left(\frac{|y|}{R}\right)\left|\partial_{x_j}\left(\Phi\left(\frac{|x|}{R}\right)\right)\right|\left|\partial_\tau\left(\Phi\left(\frac{|\tau|}{R^2}\right)\right)\right|,
\end{eqnarray*}
on $\mathcal{C}$. By letting
$$\widetilde{x}=\frac{x}{R},\qquad\widetilde{y}=\frac{y}{R},\qquad\widetilde{\tau}=\frac{\tau}{R^2},$$
we conclude that
\begin{eqnarray*}
\left|\Delta_{\mathbb{H}}\varphi_R(\eta)\right|&\leq&{R^{-2}}\left|\Delta_{\widetilde{x}}\Phi\left(|\widetilde{x}|\right)\right|\Phi\left(|\widetilde{y}|\right)\Phi\left(|\widetilde{\tau}|\right)+ \Phi\left(|\widetilde{x}|\right){R^{-2}}\left|\Delta_{\widetilde{y}}\Phi\left(|\widetilde{y}|\right)\right|\Phi\left(|\widetilde{\tau}|\right)\\
&{}&+\,4{R^{2}}(|\widetilde{x}|^2+|{\widetilde{y}}|^2)\Phi\left(|\widetilde{x}|\right)\Phi\left(|\widetilde{y}|\right){R^{-4}}\left|\partial^2_{\widetilde{\tau}}\Phi\left(|\widetilde{\tau}|\right)\right| \\
&{}&+\,4\sum_{j=1}^{n}{R}|\widetilde{x}_j|\Phi\left(|\widetilde{x}|\right){R^{-1}}\left|\partial_{{\widetilde{y}}_j}\Phi\left(|\widetilde{y}|\right)\right|{R^{-2}}\left|\partial_{\widetilde{\tau}}\Phi\left(|\widetilde{\tau}|\right)\right| \\
&{}&+\,4\sum_{j=1}^{n}{R}|{\widetilde{y}}_j|\Phi\left(|\widetilde{y}|\right){R^{-1}}\left|\partial_{\widetilde{x}_j}\Phi\left(|\widetilde{x}|\right)\right|{R^{-2}}\left|\partial_{\widetilde{\tau}}\Phi\left(|\widetilde{\tau}|\right)\right|,
\end{eqnarray*}
on $\mathcal{C}$. Note that, as $\Phi\leq1$ and $\Phi\in C^\infty$ on $\mathcal{C}$, we can easily see that
$$\left|\Delta_{\mathbb{H}}\varphi_R(\eta)\right|\leq\,C\,{R^{-2}},\qquad\hbox{for all}\,\,\eta\in\mathcal{C},$$
and therefore
\begin{align*}I_R&\leq C\,R^{-2}\int_0^\tau\int_{\mathcal{C}}v^{m}\,|\psi_0(t,\eta)|\,d\eta\,dt\\&\leq C\,R^{-2}\tau\sup_{t\in[0,\tau]}\|v(t,\cdotp)\|_{L^\infty(\mathbb{H}^n)}^{m}\sup_{t\in[0,T)}\int_{\mathbb{H}^n}|\psi_0(t,\eta)|\,d\eta,\end{align*}
this implies, using \eqref{49}, that
\begin{equation}\label{46}
I_R\longrightarrow 0,\qquad \hbox{when}\,\,R\rightarrow+\infty.
\end{equation}
Similarly,
\begin{align*}J_R:&=\left|2\int_0^\tau\int_{\mathcal{C}}v^{m}\,\nabla_{\mathbb{H}}(\varphi_R(\eta))\nabla_{\mathbb{H}}(\psi_0(t,\eta))\,d\eta\,dt\right| \\&\leq2\int_0^\tau\int_{\mathcal{C}}v^{m}\,\left|\nabla_{\mathbb{H}}\varphi_R(\eta)\right|\left|\nabla_{\mathbb{H}}\psi_0(t,\eta)\right|\,d\eta\,dt.\end{align*}
Using \eqref{48}, we have
\begin{eqnarray*}
\left|\nabla_{\mathbb{H}}\varphi_R(\eta)\right|^2&\leq&\left|\nabla_x\varphi_1(x)\right|^2\varphi^2_1(y)\varphi^2_2(\tau)+\,\varphi^2_1(x)\left|\nabla_y\varphi_1(y)\right|^2\varphi^2_2(\tau)\\
&{}&+\,4\left(|x|^2+|y|^2\right)\varphi^2_1(x)\varphi^2_1(y)\left|\partial_\tau\varphi_2(\tau)\right|^2,
\end{eqnarray*}
on $\mathcal{C}$. Substituting $\varphi_1$ and $\varphi_2$  we get
\begin{eqnarray*}
\left|\nabla_{\mathbb{H}}\varphi_R(\eta)\right|^2&\leq&\left|\nabla_x\Phi\left(\frac{|x|}{R}\right)\right|^2\Phi^2\left(\frac{|y|}{R}\right)\Phi^2\left(\frac{|\tau|}{R^2}\right)+\,\Phi^2\left(\frac{|x|}{R}\right)\left|\nabla_y\Phi\left(\frac{|y|}{R}\right)\right|^2\Phi^2\left(\frac{|\tau|}{R^2}\right)\\
&{}&+\,{4}\left(|x|^2+|y|^2\right)\Phi^2\left(\frac{|x|}{R}\right)\Phi^2\left(\frac{|y|}{R}\right)\left|\partial_\tau\Phi\left(\frac{|\tau|}{R^2}\right)\right|^2,
\end{eqnarray*}
on $\mathcal{C}$. By letting
$$\widetilde{x}=\frac{x}{R},\qquad\widetilde{y}=\frac{y}{R},\qquad\widetilde{\tau}=\frac{\tau}{R^2},$$
we conclude that
\begin{eqnarray*}
\left|\nabla_{\mathbb{H}}\varphi_R(\eta)\right|^2&\leq&{R^{-2}}\left|\nabla_{\widetilde{x}}\Phi\left(|\widetilde{x}|\right)\right|^2\Phi^2\left(|\widetilde{y}|\right)\Phi^2\left(|\widetilde{\tau}|\right)+ \,\Phi^2\left(|\widetilde{x}|\right){R^{-2}}\left|\nabla_{\widetilde{y}}\Phi\left(|\widetilde{y}|\right)\right|^2\Phi^2\left(|\widetilde{\tau}|\right)\\
&{}&+\,{4}{R^{2}}\left(|\widetilde{x}|^2+|\widetilde{y}|^2\right)\Phi^2\left(|\widetilde{x}|\right)\Phi^2\left(|\widetilde{y}|\right){R^{-4}}\left|\partial_{\widetilde{\tau}}\Phi\left(|\widetilde{\tau}|\right)\right|^2,
\end{eqnarray*}
on $\mathcal{C}$. Note that, as $\Phi\leq1$ and $\Phi\in C^\infty$ on $\mathcal{C}$, we can easily see that
$$\left|\nabla_{\mathbb{H}}\varphi(\eta)\right|\leq\,C\,{R^{-1}},\qquad\hbox{for all}\,\,\eta\in\mathcal{C},$$
and therefore
\begin{align*}J_R&\leq C\,R^{-1}\int_0^\tau\int_{\mathcal{C}}v^{m}\,|\nabla_{\mathbb{H}}\psi_0(t,\eta)|\,d\eta\,dt\\&\leq C\,R^{-1}\tau\sup_{t\in[0,\tau]}\|v(t,\cdotp)\|_{L^\infty(\mathbb{H}^n)}^{m}\sup_{t\in[0,T)}\int_{\mathbb{H}^n}|\nabla_{\mathbb{H}}\psi_0(t,\eta)|\,d\eta,\end{align*}
this implies, using \eqref{49}, that
\begin{equation}\label{47}
J_R\longrightarrow 0,\qquad \hbox{when}\,\,R\rightarrow+\infty.
\end{equation}
Finally, letting $R\longrightarrow+\infty$ in \eqref{50} and using \eqref{49}, \eqref{46},\eqref{47} together with Lebesgue's dominated convergence theorem we conclude the result.
\end{proof}
\begin{proof}[Proof of Theorem \ref{theo1}]
(i) The proof is by contradiction. Suppose that $v$ is a nonnegative global weak solution of (\ref{2}), then, for all $T\gg1$, we have
\begin{eqnarray*}
&{}&\int_{\mathbb{H}^n}v(T,\eta)\psi(T,\eta)\,d\eta-\int_{\mathbb{H}^n}v(0,\eta)\psi(0,\eta)\,d\eta\\
&{}&\qquad\qquad=\int_0^T\int_{\mathbb{H}^n}v^\sigma\psi(t,\eta)\,d\eta\,dt+\int_0^T\int_{\mathbb{H}^n}v^{m}\,\Delta_{\mathbb{H}}\psi(t,\eta)\,d\eta\,dt\\&{}&\qquad\qquad+\int_0^T\int_{\mathbb{H}^n}v\psi_t(t,\eta)\,d\eta\,dt
\end{eqnarray*}
for all compactly supported $\psi\in C^{1,2}_{t,x}([0,\infty)\times\mathbb{H}^n)$.

We choose
$$\psi(t,\eta):= \varphi^\ell(\eta) \varphi^\ell_3(t):=\varphi_1^\ell(x)\varphi_1^\ell(y)\varphi_2^\ell(\tau) \varphi^\ell_3(t),$$
 with
$$\varphi_1(x):=\Phi\left(\frac{|x|}{T^{\alpha}}\right),\quad\varphi_1(y):=\Phi\left(\frac{|y|}{T^{\alpha}}\right),\quad\varphi_2(\tau):=\Phi\left(\frac{|\tau|}{T^{2\alpha}}\right),\quad\varphi_3(t):=\Phi\left(\frac{t}{T}\right),$$
where $\alpha=\frac{\sigma-m}{2(\sigma-1)}>0$, $\ell\gg 1$, and $\Phi$ is a smooth nonnegative non-increasing function such that
\[
\Phi(r)=\left\{\begin {array}{ll}
\displaystyle{1}&\displaystyle{\quad\text{if }0\leq r\leq 1/2,}\\\\
\displaystyle{\searrow}&\displaystyle{\quad\text{if }1/2\leq r\leq 1,}\\\\
\displaystyle{0}&\displaystyle{\quad\text {if }r\geq 1.}
\end {array}\right.
\]
Then
\begin{eqnarray}\label{9}
\int_0^T\int_{\mathcal{B}}v^\sigma\psi(t,\eta)\,d\eta\,dt+\int_{\mathcal{B}}v_0(\eta)\varphi^\ell(\eta)\,d\eta&=&-\int_0^T\int_{\mathcal{C}}v^{m}\,\varphi^\ell_3(t)\,\Delta_{\mathbb{H}}\varphi^\ell(\eta)\,d\eta\,dt\nonumber\\
&{}&-\int_{\frac{T}{2}}^T\int_{\mathcal{B}}v\,\varphi^\ell(\eta)\partial_t(\varphi^\ell_3(t))\,d\eta\,dt\nonumber\\
&=&I_1+I_2,
\end{eqnarray}
where
$$\mathcal{B}=\{\eta=(x,y,\tau)\in\mathbb{H}^n;\,\,|x|^2,|y|^2,|\tau|\leq T^{2\alpha}\},$$ and $$\mathcal{C}=\{\eta=(x,y,\tau)\in\mathbb{H}^n;\,\,\frac{T^{\alpha}}{2}\leq |x|,|y|\leq T^{\alpha},\,\frac{T^{2\alpha}}{2}\leq|\tau|\leq T^{2\alpha}\}.$$
Let us start to estimate $I_1$. As $\sigma>m$, using the following Young's inequality
$$ab\leq\frac{1}{4}a^{\frac{\sigma}{m}}+C\,b^{\frac{\sigma}{\sigma-m}},$$
we have
\begin{eqnarray}\label{10}
I_1&\leq&\int_0^T\int_{\mathcal{C}}v^{m}\,\varphi^\ell_3(t)\,\left|\Delta_{\mathbb{H}}\varphi^\ell(\eta)\right|\,d\eta\,dt\nonumber\\
&=&\int_0^T\int_{\mathcal{C}}v^{m}\,\psi^{\frac{m}{\sigma}}(t,\eta)\psi^{-\frac{m}{\sigma}}(t,\eta)\varphi^\ell_3(t)\,\left|\Delta_{\mathbb{H}}\varphi^\ell(\eta)\right|\,d\eta\,dt\nonumber\\
&\leq&\frac{1}{4}\int_0^T\int_{\mathcal{B}}v^\sigma\psi(t,\eta)\,d\eta\,dt\nonumber\\&+&\,C\,\int_0^T\int_{\mathcal{C}}\psi^{-\frac{m}{\sigma-m}}(t,\eta)\varphi^{\frac{\ell \sigma}{\sigma-m}}_3(t)\,\left|\Delta_{\mathbb{H}}\varphi^\ell(\eta)\right|^{\frac{\sigma}{\sigma-m}}\,d\eta\,dt.
\end{eqnarray}
To estimate $I_2$, using the following Young's inequality
$$ab\leq\frac{1}{4}a^{\sigma}+C\,b^{\frac{\sigma}{\sigma-1}},$$
and the fact that $\sigma>1$, we have
\begin{eqnarray}\label{11}
I_2&\leq&\int_0^T\int_{\mathcal{B}}v\,\varphi^\ell(\eta)\left|\partial_t(\varphi^\ell_3(t))\right|\,d\eta\,dt\nonumber\\
&=&\int_0^T\int_{\mathcal{C}}v\,\psi^{\frac{1}{\sigma}}(t,\eta)\psi^{-\frac{1}{\sigma}}(t,\eta)\varphi^\ell(\eta)\left|\partial_t(\varphi^\ell_3(t))\right|\,d\eta\,dt\nonumber\\
&\leq&\frac{1}{4}\int_0^T\int_{\mathcal{B}}v^\sigma\psi(t,\eta)\,d\eta\,dt\nonumber\\&+&\,C\,\int_0^T\int_{\mathcal{C}}\psi^{-\frac{1}{\sigma-1}}(t,\eta)\varphi^{\frac{\ell \sigma}{\sigma-1}}(\eta)\,\left|\partial_t\varphi_3^\ell(t)\right|^{\frac{\sigma}{\sigma-1}}\,d\eta\,dt.
\end{eqnarray}
Inserting \eqref{10}-\eqref{11} into \eqref{9}, we arrive at
\begin{eqnarray}\label{16}
\frac{1}{2}\int_0^T\int_{\mathcal{B}}v^\sigma\psi(t,\eta)\,d\eta\,dt&+&\int_{\mathcal{B}}v_0(\eta)\varphi^\ell(\eta)\,d\eta
\nonumber\\&\leq&C\,\int_0^T\int_{\mathcal{C}}\psi^{-\frac{m}{\sigma-m}}(t,\eta)\varphi^{\frac{\ell \sigma}{\sigma-m}}_3(t)\,\left|\Delta_{\mathbb{H}}\varphi^\ell(\eta)\right|^{\frac{\sigma}{\sigma-m}}\,d\eta\,dt\nonumber\\
&{}&\,+\,C\,\int_0^T\int_{\mathcal{C}}\psi^{-\frac{1}{\sigma-1}}(t,\eta)\varphi^{\frac{\ell \sigma}{\sigma-1}}(\eta)\,\left|\partial_t\varphi_3^\ell(t)\right|^{\frac{\sigma}{\sigma-1}}\,d\eta\,dt\nonumber\\
&=&C\,\int_0^T\int_{\mathcal{C}}\varphi_3^{\ell}(t)\,\varphi^{-\frac{\ell m}{\sigma-m}}(\eta)\left|\Delta_{\mathbb{H}}\varphi^\ell(\eta)\right|^{\frac{\sigma}{\sigma-m}}\,d\eta\,dt\nonumber\\
&{}&\,+\,C\,\int_0^T\int_{\mathcal{C}}\varphi^\ell(\eta)\,\varphi_3^{-\frac{\ell}{\sigma-1}}(t)\,\left|\partial_t\varphi_3^\ell(t)\right|^{\frac{\sigma}{\sigma-1}}\,d\eta\,dt\nonumber\\
&=&J_1+J_2.
\end{eqnarray}
Let us estimate $J_2$. As $\partial_t\varphi_3^\ell(t)=\ell\varphi_3^{\ell-1}(t)\partial_t\varphi_3(t)$, we have
\begin{eqnarray*}
J_2&\leq& C\,\int_{\mathcal{C}}\varphi^\ell(\eta)\,d\eta\int_0^T\varphi_3^{\ell-\frac{\sigma}{\sigma-1}}(t)\,\left|\partial_t\varphi_3(t)\right|^{\frac{\sigma}{\sigma-1}}\,dt\\
&=& C\,\int_{\mathcal{C}}\varphi^\ell(\eta)\,d\eta\int_0^T\Phi^{\ell-\frac{\sigma}{\sigma-1}}\left(\frac{t}{T}\right)\,\left|\partial_t\Phi\left(\frac{t}{T}\right)\right|^{\frac{\sigma}{\sigma-1}}\,dt.
\end{eqnarray*}
Letting
$$\widetilde{x}=\frac{x}{T^\alpha},\qquad\widetilde{y}=\frac{y}{T^\alpha},\qquad\widetilde{\tau}=\frac{\tau}{T^{2\alpha}},\qquad \widetilde{t}=\frac{t}{T},$$
and using the fact that $\varphi\leq1$ and meas$(\mathcal{C})= C\,T^{\alpha Q}$, we get
\begin{equation}\label{12}
J_2\leq C\,T^{\alpha Q-\frac{\sigma}{\sigma-1}+1}\int_0^1\Phi^{\ell-\frac{\sigma}{\sigma-1}}(\tilde{t})\,\left|\Phi^{\prime}(\tilde{t})\right|^{\frac{\sigma}{\sigma-1}}\,d\widetilde{t}\leq C\,T^{\alpha Q-\frac{\sigma}{\sigma-1}+1}.
\end{equation}
To estimate $J_1$, using \eqref{40}, we have
\begin{eqnarray*}
\left|\Delta_{\mathbb{H}}\varphi^\ell(\eta)\right|&=&\left|\Delta_{\mathbb{H}}\left(\varphi_1^\ell(x)\varphi_1^\ell(y)\varphi_2^\ell(\tau) \right)\right|\\
&\leq&\left|\Delta_x\varphi_1^\ell(x)\right|\varphi_1^\ell(y)\varphi_2^\ell(\tau)\\
&{}&+\,\varphi_1^\ell(x)\left|\Delta_y\varphi_1^\ell(y)\right|\varphi_2^\ell(\tau)\\
&{}&+\,4(|x|^2+|y|^2)\varphi_1^\ell(x)\varphi_1^\ell(y)\left|\partial_\tau^2\varphi_2^\ell(\tau)\right| \\
&{}&+4\sum_{j=1}^{n}|x_j|\varphi_1^\ell(x)\left|\partial_{y_j}\varphi_1^\ell(y)\right|\left|\partial_\tau\varphi_2^\ell(\tau)\right| \\
&{}&+4\sum_{j=1}^{n}|y_j|\varphi_1^\ell(y)\left|\partial_{x_j}\varphi_1^\ell(x)\right|\left|\partial_\tau\varphi_2^\ell(\tau)\right|,
\end{eqnarray*}
on $\mathcal{C}$. So
\begin{eqnarray*}
\left|\Delta_{\mathbb{H}}\varphi^\ell(\eta)\right|&\leq&\left[\ell(\ell-1)\varphi_1^{\ell-2}(x)|\nabla_x\varphi_1(x)|^2+\ell\varphi_1^{\ell-1}(x)|\Delta_x\varphi_1(x)|\right]\varphi_1^\ell(y)\varphi_2^\ell(\tau)\\
&{}&+\,\varphi_1^\ell(x)\left[\ell(\ell-1)\varphi_1^{\ell-2}(y)|\nabla_y\varphi_1(y)|^2+\ell\varphi_1^{\ell-1}(y)|\Delta_y\varphi_1(y)|\right]\varphi_2^\ell(\tau)\\
&{}&+\,4(|x|^2+|y|^2)\varphi_1^\ell(x)\varphi_1^\ell(y)\left[\ell(\ell-1)\varphi_2^{\ell-2}(\tau)|\partial_\tau\varphi_2(\tau)|^2+\ell\varphi_2^{\ell-1}(\tau)|\partial^2_\tau\varphi_2(\tau)|\right] \\
&{}&+\,4\sum_{j=1}^{n}|x_j|\varphi_1^\ell(x)\left[\ell\varphi_1^{\ell-1}(y)\left|\partial_{y_j}\varphi_1(y)\right|\right]\left[\ell\varphi_2^{\ell-1}(\tau)\left|\partial_\tau\varphi_2(\tau)\right|\right] \\
&{}&+\,4\sum_{j=1}^{n}|y_j|\varphi_1^\ell(y)\left[\ell\varphi_1^{\ell-1}(x)\left|\partial_{x_j}\varphi_1(x)\right|\right]\left[\ell\varphi_2^{\ell-1}(\tau)\left|\partial_\tau\varphi_2(\tau)\right|\right],
\end{eqnarray*}
on $\mathcal{C}$. Substituting $\varphi_1$ and $\varphi_2$  we get
\small\begin{align*}
&\left|\Delta_{\mathbb{H}}\varphi^\ell(\eta)\right|\leq\left[\ell(\ell-1)\Phi^{\ell-2}\left(\frac{|x|}{T^\alpha}\right)\left|\nabla_x\Phi\left(\frac{|x|}{T^\alpha}\right)\right|^2 +\ell\Phi^{\ell-1}\left(\frac{|x|}{T^\alpha}\right)\left|\Delta_x\Phi\left(\frac{|x|}{T^\alpha}\right)\right|\right]\Phi^\ell\left(\frac{|y|}{T^\alpha}\right)\Phi^\ell\left(\frac{|\tau|}{T^{2\alpha}}\right)\\
&+\,\Phi^\ell\left(\frac{|x|}{T^\alpha}\right)\left[\ell(\ell-1)\Phi^{\ell-2}\left(\frac{|y|}{T^\alpha}\right)\left|\nabla_y\Phi\left(\frac{|y|}{T^\alpha}\right)\right|^2 +\ell\Phi^{\ell-1}\left(\frac{|y|}{T^\alpha}\right)\left|\Delta_y\Phi\left(\frac{|y|}{T^\alpha}\right)\right|\right]\Phi^\ell\left(\frac{|\tau|}{T^{2\alpha}}\right)\\
&+\,4(|x|^2+|y|^2)\Phi^\ell\left(\frac{|x|}{T^\alpha}\right)\Phi^\ell\left(\frac{|y|}{T^\alpha}\right)\left[\ell(\ell-1)\Phi^{\ell-2}\left(\frac{|\tau|}{T^{2\alpha}}\right)\left|\partial_\tau\Phi\left(\frac{|\tau|}{T^{2\alpha}}\right)\right|^2+\ell\Phi^{\ell-1}\left(\frac{|\tau|}{T^{2\alpha}}\right)\left|\partial^2_\tau\Phi\left(\frac{|\tau|}{T^{2\alpha}}\right)\right|\right] \\
&+\,4\sum_{j=1}^{n}|x_j|\Phi^\ell\left(\frac{|x|}{T^\alpha}\right)\left[\ell\Phi^{\ell-1}\left(\frac{|y|}{T^\alpha}\right)\left|\partial_{y_j}\Phi\left(\frac{|y|}{T^\alpha}\right)\right|\right]\left[\ell\Phi^{\ell-1}\left(\frac{|\tau|}{T^{2\alpha}}\right)\left|\partial_\tau\Phi\left(\frac{|\tau|}{T^{2\alpha}}\right)\right|\right] \\
&+\,4\sum_{j=1}^{n}|y_j|\Phi^\ell\left(\frac{|y|}{T^\alpha}\right)\left[\ell\Phi^{\ell-1}\left(\frac{|x|}{T^\alpha}\right)\left|\partial_{x_j}\Phi\left(\frac{|x|}{T^\alpha}\right)\right|\right]\left[\ell\Phi^{\ell-1}\left(\frac{|\tau|}{T^{2\alpha}}\right)\left|\partial_\tau\Phi\left(\frac{|\tau|}{T^{2\alpha}}\right)\right|\right],
\end{align*}
on $\mathcal{C}$. By letting
$$\widetilde{x}=\frac{x}{T^\alpha},\qquad\widetilde{y}=\frac{y}{T^\alpha},\qquad\widetilde{\tau}=\frac{\tau}{T^{2\alpha}}.$$
we conclude that
\begin{align*}
&\left|\Delta_{\mathbb{H}}\varphi^\ell(\eta)\right|\leq\left[\ell(\ell-1)\Phi^{\ell-2}(|\widetilde{x}|){T^{-2\alpha}} \left|\nabla_{\widetilde{x}}\Phi(|\widetilde{x}|)\right|^2+\ell\Phi^{\ell-1}(|\widetilde{x}|){T^{-2\alpha}}\left|\Delta_{\widetilde{x}}\Phi(|\widetilde{x}|)\right|\right]\Phi^\ell(|\widetilde{y}|)\Phi^\ell(|\widetilde{\tau}|)\\
&+\,\Phi^\ell(|\widetilde{x}|)\left[\ell(\ell-1)\Phi^{\ell-2}(|\widetilde{y}|){T^{-2\alpha}}\left|\nabla_{\widetilde{y}}\Phi(|\widetilde{y}|)\right|^2 +\ell\Phi^{\ell-1}(|\widetilde{y}|){T^{-2\alpha}}\left|\Delta_{\widetilde{y}}\Phi(|\widetilde{y}|)\right|\right]\Phi^\ell(|\widetilde{\tau}|)\\
&+\,4\,{T^{2\alpha}}(|\widetilde{x}|^2+|\widetilde{y}|^2)\Phi^\ell(|\widetilde{x}|)\Phi^\ell(|\widetilde{y}|)\left[\ell(\ell-1)\Phi^{\ell-2}(|\widetilde{\tau}|){T^{-4\alpha}}\left|\partial_{\widetilde{\tau}}\Phi(|\widetilde{\tau}|)\right|^2 +\ell\Phi^{\ell-1}(|\widetilde{\tau}|){T^{-4\alpha}}\left|\partial^2_{\widetilde{\tau}}\Phi(|\widetilde{\tau}|)\right|\right] \\
&+\,4\sum_{j=1}^{n}{T^{\alpha}}|\widetilde{x}_j|\Phi^\ell(|\widetilde{x}|)\left[\ell\Phi^{\ell-1}(|\widetilde{y}|){T^{-\alpha}} \left|\partial_{{\widetilde{y}}_j}\Phi(|\widetilde{y}|)\right|\right]\left[\ell\Phi^{\ell-1}(|\widetilde{\tau}|){T^{-2\alpha}}\left|\partial_{\widetilde{\tau}}\Phi(|\widetilde{\tau}|)\right|\right] \\
&+\,4\sum_{j=1}^{n}{T^{\alpha}}|\widetilde{y}_j|\Phi^\ell(|\widetilde{y}|)\left[\ell\Phi^{\ell-1}(|\widetilde{x}|){T^{-\alpha}}\left|\partial_{{\widetilde{x}}_j} \Phi(|\widetilde{x}|)\right|\right]\left[\ell\Phi^{\ell-1}(|\widetilde{\tau}|){T^{-2\alpha}}\left|\partial_{\widetilde{\tau}}\Phi(|\widetilde{\tau}|)\right|\right],
\end{align*}
on $\mathcal{C}$. Note that, as
$$\Phi\leq1\,\Rightarrow \Phi^{\ell}\leq \Phi^{\ell-1}\leq \Phi^{\ell-2},$$
we can easily see that
$$\left|\Delta_{\mathbb{H}}\varphi^\ell(\eta)\right|\leq\,C\,{T^{-2\alpha}}\left[\Phi^\ell(|\widetilde{x}|)\Phi^\ell(|\widetilde{y}|)\Phi^\ell(|\widetilde{\tau}|)\right]^{\ell-2},\qquad\hbox{for all}\,\,\eta\in\mathcal{C},$$
and therefore, using the fact that $\varphi_3\leq1$, we conclude that
\begin{eqnarray}\label{13}
J_1&=&C\,\int_0^T\int_{\mathcal{C}}\varphi_3^{\ell}(t)\,\varphi^{-\frac{\ell m}{\sigma-m}}(\eta)\left|\Delta_{\mathbb{H}}\varphi^\ell(\eta)\right|^{\frac{\sigma}{\sigma-m}}\,d\eta\,dt\nonumber\\
&\leq&C\,T^{-\frac{2\alpha\sigma}{\sigma-m}}\int_0^T\varphi_3^{\ell}(t)\,dt\int_{\mathcal{\widetilde{C}}} \left[\Phi(|\widetilde{x}|)\Phi(|\widetilde{y}|)\Phi(|\widetilde{\tau}|)\right]^{-\frac{\ell m}{\sigma-m}}\left[\Phi^\ell(|\widetilde{x}|)\Phi^\ell(|\widetilde{y}|)\Phi^\ell(|\widetilde{\tau}|)\right]^{\frac{\sigma(\ell-2)}{\sigma-m}}T^{\alpha Q}\,d\widetilde{\eta}\nonumber\\
&\leq&C\,T^{-\frac{2\alpha\sigma}{\sigma-m}+1+\alpha Q}\int_{\mathcal{\widetilde{C}}} \left[\Phi(|\widetilde{x}|)\Phi(|\widetilde{y}|)\Phi(|\widetilde{\tau}|)\right]^{\frac{\sigma\ell(\ell-2-\frac{m}{\sigma})}{\sigma-m}}\,d\widetilde{\eta}\nonumber\\
&\leq&C\,T^{-\frac{2\alpha\sigma}{\sigma-m}+1+\alpha Q},
\end{eqnarray}
where we have used the fact that $\ell\gg1$.\\
Combining \eqref{16}-\eqref{13} and taking into account that $\alpha=\frac{\sigma-m}{2(\sigma-1)}$, we get
\begin{equation}\label{17}
\frac{1}{2}\int_0^T\int_{\mathcal{B}}v^\sigma\psi(t,\eta)\,d\eta\,dt+\int_{\mathcal{B}}v_0(\eta)\varphi^\ell(\eta)\,d\eta\leq C\,T^{\frac{\sigma-m}{2(\sigma-1)} Q-\frac{\sigma}{\sigma-1}+1}.
\end{equation}
If $\sigma< m+\frac{2}{Q}$, we can easily see that $\frac{\sigma-m}{2(\sigma-1)} Q-\frac{\sigma}{\sigma-1}+1<0$, and then, using the monotone convergence theorem and the fact that $\psi(t,\eta)\rightarrow 1$ as $T\rightarrow\infty$, we conclude that
$$
0<\int_{\mathbb{H}^n}v_0(\eta)\,d\eta\leq \frac{1}{2}\int_0^\infty\int_{\mathbb{H}^n}v^\sigma(x,t)\,d\eta\,dt +\int_{\mathbb{H}^n}v_0(\eta)\,d\eta\leq0;
$$
contradiction.\\
For the critical case $\sigma= m+\frac{2}{Q}$, we can see first, using again \eqref{17} and letting $T\rightarrow\infty$, that
$$v\in L^\sigma((0,\infty)\times\mathbb{H}^n),$$
which implies that
\begin{eqnarray}\label{18}
\lim_{T\rightarrow\infty}\int_0^T\int_{\mathcal{C}}v^\sigma\psi(t,\eta)\,d\eta\,dt&=&\lim_{T\rightarrow\infty}\int_0^T\int_{\mathcal{B}}v^\sigma\psi(t,\eta)\,d\eta\,dt-\lim_{T\rightarrow\infty}\int_0^T\int_{\mathcal{C}_0}v^\sigma\psi(t,\eta)\,d\eta\,dt\nonumber\\
&=&\int_0^\infty\int_{\mathbb{H}^n}v^\sigma(x,t)\,d\eta\,dt-\int_0^\infty\int_{\mathbb{H}^n}v^\sigma(x,t)\,d\eta\,dt\nonumber\\
&=&0,
\end{eqnarray}
where
\begin{equation}\label{C_0}
\mathcal{C}_0=\{\eta=(x,y,\tau)\in\mathbb{H}^n;\,\, |x|,|y|\leq \frac{T^{\alpha}}{2},\,|\tau|\leq \frac{T^{2\alpha}}{2}\},
\end{equation}
and
\begin{eqnarray}\label{19}
\lim_{T\rightarrow\infty}\int_{\frac{T}{2}}^T\int_{\mathcal{B}}v^\sigma\psi(t,\eta)\,d\eta\,dt&=&\lim_{T\rightarrow\infty}\int_0^T\int_{\mathcal{B}}v^\sigma\psi(t,\eta)\,d\eta\,dt-\lim_{T\rightarrow\infty}\int_0^{\frac{T}{2}}\int_{\mathcal{B}}v^\sigma\psi(t,\eta)\,d\eta\,dt\nonumber\\
&=&\int_0^\infty\int_{\mathbb{H}^n}v^\sigma(x,t)\,d\eta\,dt-\int_0^\infty\int_{\mathbb{H}^n}v^\sigma(x,t)\,d\eta\,dt\nonumber\\
&=&0.
\end{eqnarray}
On the other hand, we need to use H\"{o}lder's inequality instead of Young's one in the estimations of $I_1$ and $I_2$, and to refine them. Indeed,
\begin{eqnarray}\label{14}
I_1&\leq&\int_0^T\int_{\mathcal{C}}v^{m}\,\varphi^\ell_3(t)\,\left|\Delta_{\mathbb{H}}\varphi^\ell(\eta)\right|\,d\eta\,dt\nonumber\\
&=&\int_0^T\int_{\mathcal{C}}v^{m}\,\psi^{\frac{m}{\sigma}}(t,\eta)\psi^{-\frac{m}{\sigma}}(t,\eta)\varphi^\ell_3(t)\,\left|\Delta_{\mathbb{H}}\varphi^\ell(\eta)\right|\,d\eta\,dt\nonumber\\
&\leq&\left(\int_0^T\int_{\mathcal{C}}v^\sigma\psi(t,\eta)\,d\eta\,dt\right)^{\frac{\sigma}{m}}\left(\int_0^T\int_{\mathcal{C}}\psi^{-\frac{m}{\sigma-m}}(t,\eta)\varphi^{\frac{\ell \sigma}{\sigma-m}}_3(t)\,\left|\Delta_{\mathbb{H}}\varphi^\ell(\eta)\right|^{\frac{\sigma}{\sigma-m}}\,d\eta\,dt\right)^{\frac{\sigma-m}{\sigma}}\nonumber\\
&=&\left(\int_0^T\int_{\mathcal{C}}v^\sigma\psi(t,\eta)\,d\eta\,dt\right)^{\frac{\sigma}{m}}J_1^{\frac{\sigma-m}{\sigma}},
\end{eqnarray}
and
\begin{eqnarray}\label{15}
I_2&\leq&\int_{\frac{T}{2}}^T\int_{\mathcal{B}}v\,\varphi^\ell(\eta)\left|\partial_t(\varphi^\ell_3(t))\right|\,d\eta\,dt\nonumber\\
&=&\int_{\frac{T}{2}}^T\int_{\mathcal{C}}v\,\psi^{\frac{1}{\sigma}}(t,\eta)\psi^{-\frac{1}{\sigma}}(t,\eta)\varphi^\ell(\eta)\left|\partial_t(\varphi^\ell_3(t))\right|\,d\eta\,dt\nonumber\\
&\leq&\left(\int_{\frac{T}{2}}^T\int_{\mathcal{B}}v^\sigma\psi(t,\eta)\,d\eta\,dt\right)^{\frac{1}{\sigma}}\left(\int_0^T\int_{\mathcal{C}}\psi^{-\frac{1}{\sigma-1}}(t,\eta)\varphi^{\frac{\ell \sigma}{\sigma-1}}(\eta)\,\left|\partial_t\varphi_3^\ell(t)\right|^{\frac{\sigma}{\sigma-1}}\,d\eta\,dt\right)^{\frac{\sigma-1}{\sigma}}\nonumber\\
&\leq&\left(\int_{\frac{T}{2}}^T\int_{\mathcal{B}}v^\sigma\psi(t,\eta)\,d\eta\,dt\right)^{\frac{1}{\sigma}}J_2^{\frac{\sigma-1}{\sigma}}.
\end{eqnarray}
Inserting \eqref{14}-\eqref{15} into \eqref{9}, we arrive at
$$
\int_{\mathcal{B}}v_0(\eta)\varphi^\ell(\eta)\,d\eta\leq\left(\int_0^T\int_{\mathcal{C}}v^\sigma\psi(t,\eta)\,d\eta\,dt\right)^{\frac{\sigma}{m}}J_1^{\frac{\sigma-m}{\sigma}}+\left(\int_{\frac{T}{2}}^T\int_{\mathcal{B}}v^\sigma\psi(t,\eta)\,d\eta\,dt\right)^{\frac{1}{\sigma}}J_2^{\frac{\sigma-1}{\sigma}}.
$$
By letting
$$\widetilde{x}=\frac{x}{T^\alpha},\qquad\widetilde{y}=\frac{y}{T^\alpha},\qquad\widetilde{\tau}=\frac{\tau}{T^{2\alpha}},\qquad \widetilde{t}=\frac{t}{T},$$
inside $J_1$ and $J_2$, using their estimates and that $\sigma= m+\frac{2}{Q}$, we obtain
\begin{equation}\label{A1}
\int_{\mathcal{B}}v_0(\eta)\varphi^\ell(\eta)\,d\eta\leq C\,\left(\int_0^T\int_{\mathcal{C}}v^\sigma\psi(t,\eta)\,d\eta\,dt\right)^{\frac{\sigma}{m}}+\,C\,\left(\int_{\frac{T}{2}}^T\int_{\mathcal{B}}v^\sigma\psi(t,\eta)\,d\eta\,dt\right)^{\frac{1}{\sigma}}.
\end{equation}
Finally, using \eqref{18},\eqref{19}, \eqref{A1}, the dominated convergence theorem, and the fact that $\psi(t,\eta)\rightarrow 1$ as $T\rightarrow\infty$, we conclude that
$$
0<\int_{\mathbb{H}^n}v_0(\eta)\,d\eta\leq 0;
$$
contradiction.\\

(ii) As $T>1$, we have
\begin{eqnarray*}
\int_{\mathcal{B}}v_0(\eta)\psi(0,\eta)\,d\eta&=&\int_{\mathcal{B}}v_0(\eta)\varphi^\ell(\eta)\,d\eta
\geq \int_{\mathcal{C}_0}v_0(\eta)\varphi^\ell(\eta)\,d\eta\\
&=&\int_{\mathcal{C}_0}v_0(\eta)\,d\eta
\geq \varepsilon \int_{\mathcal{C}_0}(1+|\eta|_{_{\mathbb{H}}}^2)^{-\gamma/2}\,d\eta\\
&\geq& \varepsilon\,C \int_{\mathcal{C}_0}\left(\frac{T^{2\alpha}}{2}+ \frac{T^{2\alpha}}{2}\right)^{-\gamma/2}\,d\eta\\
&=&\varepsilon\,C T^{-\gamma\alpha}\,\hbox{meas}(\mathcal{C}_0)\\
&=&\varepsilon\,C T^{\alpha(Q-\gamma)},
\end{eqnarray*}
where $\mathcal{C}_0$ is defined in \eqref{C_0}. Therefore by repeating the same calculation as in the subcritical case (i) with $\alpha=\frac{\sigma-m}{2(\sigma-1)}$, we get
$$\varepsilon\,C T^{\frac{\sigma-m}{2(\sigma-1)}(Q-\gamma)}+\frac{1}{2}\int_0^T\int_{\mathcal{B}}v^\sigma\psi(t,\eta)\,d\eta\,dt\leq C\,T^{\frac{\sigma-m}{2(\sigma-1)} Q-\frac{\sigma}{\sigma-1}+1},$$
which implies
$$\varepsilon\,C T^{\frac{\sigma-m}{2(\sigma-1)}(Q-\gamma)}\leq C\,T^{\frac{\sigma-m}{2(\sigma-1)} Q-\frac{\sigma}{\sigma-1}+1},$$
that is
$$\varepsilon\leq C\,T^{\frac{\sigma-m}{2(\sigma-1)} \gamma-\frac{\sigma}{\sigma-1}+1}.$$
As $\sigma<m+\frac{2}{\gamma}\Longleftrightarrow {\frac{\sigma-m}{2(\sigma-1)} \gamma-\frac{\sigma}{\sigma-1}+1}<0$, then, by passing to the limit, as $T$ goes to $\infty$, we get a contradiction.

(iii) Let $m=\sigma.$ In this case, using Lemma \ref{lemma3}, we may replace, in the test function, $\varphi^\ell(\eta)$ by $\Theta(\eta)$ where $\Theta$ is defined in Lemma \ref{lemma4} with $\varepsilon=\frac{1}{4(2+Q)}$ i.e.
$$
-\Delta_{\mathbb{H}}\Theta(\eta)\leq\frac{1}{2}\Theta(\eta),\qquad\hbox{for all}\,\,\eta\in\mathbb{H}^n.
$$
Therefore, by repeating the same calculation as before, we have from \eqref{9}
\begin{align*}
&\int_0^T\int_{\mathbb{H}^n}v^\sigma\psi(t,\eta)\,d\eta\,dt+\int_{\mathbb{H}^n}v_0(\eta)\Theta(\eta)\,d\eta\leq\int_0^T\int_{\mathbb{H}^n}v^{\sigma}\,\varphi^\ell_3(t)\,(-\Delta_{\mathbb{H}})\Theta(\eta)\,d\eta\,dt \\&-\int_{\frac{T}{2}}^T\int_{\mathbb{H}^n}v\,\Theta(\eta)\partial_t(\varphi^\ell_3(t))\,d\eta\,dt\nonumber\\
&\leq\frac{1}{2}\int_0^T\int_{\mathbb{H}^n}v^{\sigma}\,\varphi^\ell_3(t)\,\Theta(\eta)\,d\eta\,dt-\int_{\frac{T}{2}}^T\int_{\mathbb{H}^n}v\,\Theta(\eta)\partial_t(\varphi^\ell_3(t))\,d\eta\,dt,
\end{align*}
which is equivalent to
$$
\frac{1}{2}\int_0^T\int_{\mathbb{H}^n}v^\sigma\psi(t,\eta)\,d\eta\,dt+\int_{\mathbb{H}^n}v_0(\eta)\Theta(\eta)\,d\eta\leq-\int_{\frac{T}{2}}^T\int_{\mathbb{H}^n}v\,\Theta(\eta)\partial_t(\varphi^\ell_3(t))\,d\eta\,dt=I_2
$$
where $I_2$ is introduced above. Then, using \eqref{11}, we get
\begin{eqnarray*}
&{}&\frac{1}{2}\int_0^T\int_{\mathbb{H}^n}v^\sigma\psi(t,\eta)\,d\eta\,dt+\int_{\mathbb{H}^n}v_0(\eta)\Theta(\eta)\,d\eta\\
&{}&\leq\frac{1}{4}\int_0^T\int_{\mathbb{H}^n}v^\sigma\psi(t,\eta)\,d\eta\,dt+\,C\,\int_0^T\int_{\mathbb{H}^n}\psi^{-\frac{1}{\sigma-1}}(t,\eta)\Theta^{\frac{\sigma}{\sigma-1}}(\eta)\,\left|\partial_t\varphi_3^\ell(t)\right|^{\frac{\sigma}{\sigma-1}}\,d\eta\,dt,
\end{eqnarray*}
i.e.
$$
\frac{1}{4}\int_0^T\int_{\mathbb{H}^n}v^\sigma\psi(t,\eta)\,d\eta\,dt+\int_{\mathbb{H}^n}v_0(\eta)\Theta(\eta)\,d\eta\\
\leq\,C\,\int_0^T\int_{\mathbb{H}^n}\psi^{-\frac{1}{\sigma-1}}(t,\eta)\Theta^{\frac{\sigma}{\sigma-1}}(\eta)\,\left|\partial_t\varphi_3^\ell(t)\right|^{\frac{\sigma}{\sigma-1}}\,d\eta\,dt,
$$
and so
\begin{eqnarray*}
\int_{\mathbb{H}^n}v_0(\eta)\Theta(\eta)\,d\eta&\leq&\,C\,\int_0^T\int_{\mathbb{H}^n}\psi^{-\frac{1}{\sigma-1}}(t,\eta)\Theta^{\frac{\sigma}{\sigma-1}}(\eta)\,\left|\partial_t\varphi_3^\ell(t)\right|^{\frac{\sigma}{\sigma-1}}\,d\eta\,dt\\
&=&C\,\int_0^T\int_{\mathbb{H}^n}\Theta(\eta)\,\varphi_3^{-\frac{\ell}{\sigma-1}}(t)\,\left|\partial_t\varphi_3^\ell(t)\right|^{\frac{\sigma}{\sigma-1}}\,d\eta\,dt.
\end{eqnarray*}
As $\partial_t\varphi_3^\ell(t)=\ell\varphi_3^{\ell-1}(t)\partial_t\varphi_3(t)$, we obtain
\begin{eqnarray*}
\int_{\mathbb{H}^n}v_0(\eta)\Theta(\eta)\,d\eta&\leq& C\,\int_{\mathbb{H}^n}\Theta(\eta)\,d\eta\int_0^T\varphi_3^{\ell-\frac{\sigma}{\sigma-1}}(t)\,\left|\partial_t\varphi_3(t)\right|^{\frac{\sigma}{\sigma-1}}\,dt\\
&\leq& C\,\int_0^T\Phi^{\ell-\frac{\sigma}{\sigma-1}}\left(\frac{t}{T}\right)\,\left|\partial_t\Phi\left(\frac{t}{T}\right)\right|^{\frac{\sigma}{\sigma-1}}\,dt.
\end{eqnarray*}
By taking $\widetilde{t}=\frac{t}{T}$, we conclude that
\begin{equation}\label{}
\int_{\mathbb{H}^n}v_0(\eta)\Theta(\eta)\,d\eta\leq C\,T^{-\frac{\sigma}{\sigma-1}+1}\int_0^1\Phi^{\ell-\frac{\sigma}{\sigma-1}}(\tilde{t})\,\left|\Phi^{\prime}(\tilde{t})\right|^{\frac{\sigma}{\sigma-1}}\,d\widetilde{t}\leq C\,T^{-\frac{\sigma}{\sigma-1}+1}.
\end{equation}
By letting $T\rightarrow\infty$ we obtain a contradiction with $v_0(\eta)\geq0,\,v_0(\eta)\not\equiv 0$. This completes the proof.
\end{proof}
Next, we shall prove the nonexistence of positive classical and weak solutions in the case of
large data by an energy-type method as performed e.g. in \cite{Friedman,Wink1}.
\begin{theorem}\label{theo03}
Let  $n\geq1$, and $1<m<\sigma$. For each $0<w\in C(\mathbb{H}^n)\cap L^\infty(\mathbb{H}^n)$, there is $B>0$ such that if $v_0=Bw$ then there are no positive global classical solutions of \eqref{2}. More precisely, there exists a $T^{*}>0$ such that
$$\sup_{\eta\in\mathbb{H}^n} v(t,\eta)\longrightarrow\infty,\qquad\hbox{as}\,\,t\rightarrow T^{*}.$$
\end{theorem}
\begin{proof} Suppose, on the contrary, that $v$ is a positive global classical solution of \eqref{2}, i.e. a positive classical solution of \eqref{2} on $[0,T]$ for all $T>0$.

Let $\Omega\subset \mathbb{H}^n$ be a Heisenberg ball with boundary $\partial\Omega$, and let $\lambda_1>0$ be the principal eigenvalue of $-\Delta_{ \mathbb{H}}$ with Dirichlet condition and $\Lambda>0$ its corresponding eigenfunction such that $\int_{\Omega}\Lambda(\eta)\,d\eta=1$ (The existence of such eigenvalue has been proved in \cite{Chen1}). In order to get a contradiction, we are going to apply the energy method. We divide our proof into three steps.\\
\noindent{\bf Step 1.} Let
$$y(t):=\int_{\Omega}v(t,\eta)\Lambda(\eta)\,d\eta,\quad t\in[0,T].$$  \\
As $v$ is a classical solution, we have
$$y\in C([0,T])\cap C^1((0,T]).$$
Using the Green's formula for Heisenberg group (see \cite{Gav, Ruzhansky2}) one can get
\begin{eqnarray*}
y^{\prime}(t)&=&\int_{\Omega}v_t(t,\eta)\Lambda(\eta)\,d\eta\\
&=&\int_{\Omega}\Delta_\mathbb{H} v^m(t,\eta)\Lambda(\eta)\,d\eta+\int_{\Omega}v^{\sigma}(t,\eta)\Lambda(\eta)\,d\eta\\
&=&-\lambda_1\int_{\Omega} v^m(t,\eta)\Lambda(\eta)\,d\eta-\int_{\partial\Omega} v^m(t,\sigma)\partial_{\nu}\Lambda(\sigma)\,d\sigma+\int_{\Omega}v^{\sigma}(t,\eta)\Lambda(\eta)\,d\eta.
\end{eqnarray*}
It follows from the Hopf type lemma on the Heisenberg group $\mathbb{H}^n$ (see \cite[Lemma 2.1]{BCutri}), that $\partial_{\nu}\Lambda\leq 0$ on $\partial\Omega$. Then we have
\begin{equation}\label{26}
y^{\prime}(t)\geq-\lambda_1\int_{\Omega} v^m(t,\eta)\Lambda(\eta)\,d\eta+\int_{\Omega}v^{\sigma}(t,\eta)\Lambda(\eta)\,d\eta.
\end{equation}
In order to apply the energy method, i.e. obtaining a differential inequality for $y(t)$, we need to estimate the right-hand side of \eqref{26}. Let $v_0=Bw$, where $0<w\in C(\mathbb{H}^n)\cap L^\infty(\mathbb{H}^n)$ and $B\gg1$ is a positive real number such that
$$B>(2\lambda_1)^{\frac{1}{\sigma-m}}\left(\int_{\Omega}w(\eta)\Lambda(\eta)\,d\eta\right)^{-1}.$$
This implies that $y_0:=y(0)>c_3$, with
$$c_3:=\left(2\lambda_1\right)^{\frac{1}{\sigma-m}}.$$
\noindent{\bf Step 2.}  We have $y(t)\geq c_3$, for all $t\in(0,T]$. Indeed, let $T_0=\inf\{0<t\leq T;\,y(t)\geq c_3\}\leq T$. Since $y$ is continuous and $y(0)>c_3$, we have $T_0>0$. We claim $T_0=T$. Otherwise, we have $y(t)>c_3$ for all $t\in(0,T_0)$ such that $y(T_0)=c_3$, i.e. particularly, $y(t)\geq c_3$ for all $t\in[0,T_0]$. On the other hand, using $m>1$ and applying the following H\"{o}lder's inequality for negative exponent (see \cite[p. 27]{Adams})
$$\int|fg|\,d\mu\geq\left(\int |f|^{r_1}\,d\mu\right)^{\frac{1}{r_1}}\left(\int |g|^{r_2}\,d\mu\right)^{\frac{1}{r_2}},\,\,\hbox{for all}\,\,r_1<0, 0<r_2<1, \,\frac{1}{r_1}+\frac{1}{r_2}=1,$$
 with $r_1=\frac{1}{1-m}$ and $r_2=\frac{1}{m}$, we have
\begin{eqnarray}\label{27}
\int_{\Omega}v^m(t,\eta)\Lambda(\eta)\,d\eta&=&\int_{\Omega}v^m(t,\eta)\Lambda^m(\eta)\Lambda^{1-m}(\eta)\,d\eta\nonumber\\
&\geq&\left(\int_{\Omega}v(t,\eta)\Lambda(\eta)\,d\eta\right)^{m}\left(\int_{\Omega}\Lambda(\eta)\,d\eta\right)^{1-m}\nonumber\\
&=&\left(\int_{\Omega}v(t,\eta)\Lambda(\eta)\,d\eta\right)^{m}\nonumber\\
&=&y^m(t),
\end{eqnarray}
 for all $t\in[0,T]$, where we have used that $\displaystyle\int_{\Omega}\Lambda(\eta)\,d\eta=1$. In addition, using again H\"{o}lder's inequality for negative exponent with $r_1=\frac{m}{m-\sigma}<0$ and $r_2=\frac{m}{\sigma}<1$, we have
\begin{eqnarray*}
\int_{\Omega}v^\sigma(t,\eta)\Lambda(\eta)\,d\eta&=&\int_{\Omega}v^\sigma(t,\eta)\Lambda^{\frac{\sigma}{m}}(\eta)\Lambda^{1-\frac{\sigma}{m}}(\eta)\,d\eta\\
&\geq& \left(\int_{\Omega}v^{m}(t,\eta)\Lambda(\eta)\,d\eta\right)^{\frac{\sigma}{m}}\left(\int_{\Omega}\Lambda(\eta)\,d\eta\right)^{\frac{m-\sigma}{m}}\\
&=& \left(\int_{\Omega}v^{m}(t,\eta)\Lambda(\eta)\,d\eta\right)^{\frac{\sigma}{m}}
\end{eqnarray*}
which implies, using \eqref{27} and $y(t)\geq c_3$, that
\begin{eqnarray}\label{28}
\int_{\Omega}v^\sigma(t,\eta)\Lambda(\eta)\,d\eta&=& \left(\int_{\Omega}v^{m}(t,\eta)\Lambda(\eta)\,d\eta\right)^{\frac{\sigma}{m}-1} \left(\int_{\Omega}v^{m}(t,\eta)\Lambda(\eta)\,d\eta\right)\nonumber\\
&\geq&(y^m(t))^{\frac{\sigma}{m}-1}\left(\int_{\Omega}v^{m}(t,\eta)\Lambda(\eta)\,d\eta\right)\nonumber\\
&=&y^{\sigma-m}(t)\int_{\Omega}v^m(t,\eta)\Lambda(\eta)\,d\eta\nonumber\\
&\geq&c_3^{\sigma-m}\int_{\Omega}v^m(t,\eta)\Lambda(\eta)\,d\eta\nonumber\\
&=&2\lambda_1\int_{\Omega}v^m(t,\eta)\Lambda(\eta)\,d\eta,
\end{eqnarray}
 for all $t\in(0,T_0]$. Therefore, by \eqref{26} and \eqref{28}, we arrive at
\begin{align*}
y^{\prime}(t)&\geq-\lambda_1\int_{\Omega}v^m(t,\eta)\Lambda(\eta)\,d\eta+2\lambda_1\int_{\Omega}v^m(t,\eta)\Lambda(\eta)\,d\eta\\& =\lambda_1\int_{\Omega}v^m(t,\eta)\Lambda(\eta)\,d\eta,\end{align*}
which implies, using \eqref{27}, that
$$y^{\prime}(t)\geq \lambda_1y^m(t)\geq0,\qquad\hbox{ for all}\,\,t\in(0,T_0],
$$
and hence
$$c_3=y(T_0)\geq y(0)=y_0>c_3;$$
contradiction.\\
\noindent{\bf Step 3.} From Step 2, we have $y(t)\geq c_3$, for all $t\in[0,T]$. This implies, using \eqref{27}-\eqref{28}, that
$$
y^{\prime}(t)\geq \lambda_1y^m(t),\qquad\hbox{ for all}\,\,t\in(0,T],
$$
so
$$y(t)\geq \left(y_0^{1-m}-(m-1)\lambda_1\,t\right)^{-\frac{1}{m-1}},\qquad\hbox{ for all}\,\,t\in[0,T].$$
Let
$$T^{*}=\frac{1}{y_0^{m-1}(m-1)\lambda_1}.$$
If $T^{*}<T$, we also get a contradiction because
$$\sup_{\eta\in\mathbb{H}^n} v(t,\eta)\geq y(t)\geq \left(y_0^{1-m}-(m-1)\lambda_1\,t\right)^{-\frac{1}{m-1}}\longrightarrow\infty,\qquad\hbox{when}\,\,t\rightarrow T^{*}.$$
If $T\leq T^{*}$, we get a contradiction by choosing from the beginning $T$ big enough, namely $T>T^{*}$.\\
 This completes the proof.
\end{proof}

\begin{theorem}\label{theo3}
Let  $n\geq1$, and $m>0$, $\sigma>1$.\\
 If $m<\sigma$, then for each $0<w\in  L^1(\mathbb{H}^n)\cap BC(\mathbb{H}^n)$, there is $B>0$ such that if $v_0=Bw$ there are no positive global weak solutions $v\in C([0,\infty);L^1(\mathbb{H}^n))\cap L^\infty_{loc}((0,\infty);L^\infty(\mathbb{H}^n))$ of \eqref{2}.\\
If $m=\sigma$, then for each $0< v_0 \in  L^1(\mathbb{H}^n)\cap BC(\mathbb{H}^n)$ there are no positive global weak solutions $v\in C([0,\infty);L^1(\mathbb{H}^n))\cap L^\infty_{loc}((0,\infty);L^\infty(\mathbb{H}^n))$ of \eqref{2}.\\
More precisely, there exists a $T^{*}>0$ such that
$$\sup_{\eta\in\mathbb{H}^n} v(t,\eta)\longrightarrow\infty,\qquad\hbox{as}\,\,t\rightarrow T^{*}.$$
\end{theorem}
\begin{proof} Suppose, on the contrary, that $v\in C([0,\infty);L^1(\mathbb{H}^n))\cap L^\infty_{loc}((0,\infty);L^\infty(\mathbb{H}^n))$ is a positive weak solution of \eqref{2} on $[0,\infty)\times\mathbb{H}^n$.\\
\noindent\underline{{\bf The case $m<\sigma$:}} Let $\Theta_1(\eta):=c_{*}\Theta(\eta)$, $\eta\in\mathbb{H}^n$, where $\Theta$ is defined in Lemma \ref{lemma4} with $\varepsilon=1$, and $c_{*}>0$ is a constant such that $\int_{\mathbb{H}^n}\Theta_1(x)\,d\eta=1$, namely $c_{*}=\left(\int_{\mathbb{H}^n}\Theta(\eta)\,d\eta\right)^{-1}$. Then
\begin{equation}\label{29}
\Delta_{\mathbb{H}}\Theta_1(\eta)\geq-\lambda\Theta_1(\eta),\qquad\hbox{for all}\,\,\eta\in\mathbb{H}^n,
\end{equation}
where $\lambda=2(2+Q)$. In order to get a contradiction, we are going to apply the energy method. We divide our proof into three steps.\\
\noindent{\bf Step 1.} Let
$$J(t):=\int_{\mathbb{H}^n}v(t,\eta)\Theta_1(\eta)\,d\eta,\quad t\geq0.$$
As $v$ is a weak solution, by Lemma \ref{lemma3}  we may choose $\psi(t,\eta)=\Theta_1(\eta)$ as a test function. Therefore, using the continuity (in time) of $v$ and \eqref{29}, we have $J\in C([0,\infty))$ and
\begin{eqnarray*}
J(\tau)-J(0)&=&\int_{\mathbb{H}^n}v(t,\eta)\Theta_1(\eta)\,d\eta-\int_{\mathbb{H}^n}v_0(\eta)\Theta_1(\eta)\,d\eta
\\&=&\int_0^\tau\int_{\mathbb{H}^n}v^\sigma\Theta_1(\eta)\,d\eta\,dt+\int_0^\tau\int_{\mathbb{H}^n}v^{m}\,\Delta_{\mathbb{H}}\Theta_1(\eta)\,d\eta\,dt\\
&\geq&\int_0^\tau\int_{\mathbb{H}^n}(v^\sigma-\lambda v^{m})\Theta_1(\eta)\,d\eta\,dt\\
&=&\int_0^\tau\int_{\mathbb{H}^n}F(v)\Theta_1(\eta)\,d\eta\,dt,
\end{eqnarray*}
for all $\tau\in[0,\infty)$, where $F(z):=z^\sigma-\lambda z^{m}$, $z>0$.\\
\noindent{\bf Step 2.}  Let $v_0=Bw$, where $0<w\in  L^1(\mathbb{H}^n)\cap BC(\mathbb{H}^n)$ and $B\gg1$ is a positive real number such that
$$B>\lambda^{\frac{1}{\sigma-m}}\left(\int_{\Omega}w(\eta)\Theta_1(\eta)\,d\eta\right)^{-1}.$$
This is equivalent to $J(0)>\lambda^{\frac{1}{\sigma-m}}$. Therefore, by the continuity of $J$, there exists $0<t_0\ll1$ sufficiently small such that $J(t)>\lambda^{\frac{1}{\sigma-m}}$ for all $0\leq t< t_0$. We claim that $J(t)>\lambda^{\frac{1}{\sigma-m}}$, for all $t\geq0$. Indeed, assume on the contrary that $J(t)\leq\lambda^{\frac{1}{\sigma-m}}$, for some $t\geq t_0$. Let $\tau_0$ be the smallest such value, this implies that
$$J(\tau)>\lambda^{\frac{1}{\sigma-m}},\quad\hbox{for all}\,\,0\leq\tau<\tau_0,\qquad\hbox{and}\quad J(\tau_0)=\lambda^{\frac{1}{\sigma-m}},$$
particularly $J(\tau)\geq\lambda^{\frac{1}{\sigma-m}}$ for all $0\leq\tau\leq\tau_0$. On the other hand, we can easily see that $F$ is convex on $(0,\infty)$ if $m\leq 1$ and on $((\frac{\lambda m(m-1)}{\sigma(\sigma-1)})^{\frac{1}{\sigma-m}},\infty)$ if $m> 1$. Therefore by using
$$J(\tau)\geq\lambda^{\frac{1}{\sigma-m}}>\max\left\{\left(\frac{\lambda m(m-1)}{\sigma(\sigma-1)}\right)^{\frac{1}{\sigma-m}};0\right\},\quad\hbox{for all}\,\,0\leq\tau\leq\tau_0,$$
Jensen's inequality and the fact that $\int_{\mathbb{H}}\Theta_1(x)\,d\eta=1$, we get
\begin{equation}\label{30}
J(\tau)\geq J(0)+\int_0^\tau F(J(t))\,dt=:G(\tau),\quad\hbox{for all}\,\,0\leq\tau\leq\tau_0.
\end{equation}
Moreover, as $F$ is positive on $(\lambda^{\frac{1}{\sigma-m}},\infty)$, we have $\int_0^{\tau_0} F(J(t))\,dt>0$, which implies
$$\lambda^{\frac{1}{\sigma-m}}=J(\tau_0)\geq J(0)+\int_0^\tau F(J(t))\,dt>J(0)>\lambda^{\frac{1}{\sigma-m}};$$
contradiction.\\
\noindent{\bf Step 3.} From Step 2, we have $$J(t)>\lambda^{\frac{1}{\sigma-m}}>(\frac{\lambda m}{\sigma})^{\frac{1}{\sigma-m}},\qquad\hbox{for all $t\geq0$}.$$
 This implies, as $F$ is increasing on $((\frac{\lambda m}{\sigma})^{\frac{1}{\sigma-m}},\infty)$ and using \eqref{30}, that $$F(J(\tau))\geq F(G(\tau))>0$$ and $$F(G(\tau))\geq F(J(0))>0,$$ i.e.
\begin{align*}G'(\tau)&=F(J(\tau))\geq F(G(\tau))\\&=G^\sigma(\tau)-\lambda G^m(\tau)\\&=G^\sigma(\tau)(1-\lambda G^{m-\sigma}(\tau)).\end{align*}
In addition, as $G(\tau)\geq J(0)$, it follows that $$1-\lambda G^{m-\sigma}(\tau)>1-\lambda J^{m-\sigma}(0)>0,$$ and so
$$\frac{G'(\tau)}{G^\sigma(\tau)}\geq 1-\lambda J^{m-\sigma}(0),\quad\hbox{for all}\,\,\tau\geq0.$$
Integrating both sides over $(0,t)$, we arrive at
$$G(t)\geq\frac{1}{\left(J^{1-\sigma}(0)-(\sigma-1)(1-\lambda J^{m-\sigma}(0))t\right)^{\sigma-1}}.$$
Let
$$T^{*}=\frac{J^{1-\sigma}(0)}{(\sigma-1)(1-\lambda J^{m-\sigma}(0))},$$ then
\begin{align*}\sup_{\eta\in\mathbb{H}^n} v(t,\eta)&\geq J(t)\geq G(t)\\&\geq\frac{1}{\left(J^{1-\sigma}(0)-(\sigma-1)(1-\lambda J^{m-\sigma}(0))t\right)^{\sigma-1}}\longrightarrow\infty,\quad\hbox{when}\,\,t\rightarrow T^{*}.\end{align*}
This completes the proof.\\
\noindent\underline{{\bf The case $m=\sigma$:}} Let $\Theta_1(\eta):=c_{*}\Theta(\eta)$, $\eta\in\mathbb{H}^n$, where $\Theta$ is defined in Lemma \ref{lemma4} with $\varepsilon=\frac{1}{4(2+Q)}$, and $c_{*}>0$ is a constant such that $\int_{\mathbb{H}^n}\Theta_1(x)\,d\eta=1$, namely $c_{*}=\left(\int_{\mathbb{H}^n}e^{-|\eta|^2_{_{\mathbb{H}}}}\,d\eta\right)^{-1}$. Then
\begin{equation}\label{32}
\Delta_{\mathbb{H}}\Theta_1(\eta)\geq-\frac{1}{2}\Theta_1(\eta),\qquad\hbox{for all}\,\,\eta\in\mathbb{H}^n.
\end{equation}
In order to get a contradiction, we are going to apply the energy method. We divide our proof into two steps.\\
\noindent{\bf Step 1.} Let
$$J(t):=\int_{\mathbb{H}^n}v(t,\eta)\Theta_1(\eta)\,d\eta,\quad t\geq0.$$
As $v$ is a weak solution, by Lemma \ref{lemma3}  we may choose $\psi(t,\eta)=\Theta_1(\eta)$ as a test function. Therefore, using the continuity (in time) of $v$ and \eqref{32}, we have $J\in C([0,\infty))$ and
\begin{eqnarray}\label{33}
J(\tau)-J(0)&=&\int_0^\tau\int_{\mathbb{H}^n}v^\sigma\Theta_1(\eta)\,d\eta\,dt+\int_0^\tau\int_{\mathbb{H}^n}v^{\sigma}\,\Delta_{\mathbb{H}}\Theta_1(\eta)\,d\eta\,dt\nonumber\\
&\geq&\frac{1}{2}\int_0^\tau\int_{\mathbb{H}^n}v^\sigma\Theta_1(\eta)\,d\eta\,dt,
\end{eqnarray}
for all $\tau\in[0,\infty)$. On the other hand, using $\sigma>1$ and applying the following H\"{o}lder's inequality for negative exponent \cite[p. 27]{Adams}
$$\int|fg|\,d\mu\geq\left(\int |f|^{r_1}\,d\mu\right)^{\frac{1}{r_1}}\left(\int |g|^{r_2}\,d\mu\right)^{\frac{1}{r_2}},\,\,\, \hbox{for all}\,\,r_1<0, 0<r_2<1, \,\frac{1}{r_1}+\frac{1}{r_2}=1,$$
 with $r_1=\frac{1}{1-\sigma}$ and $r_2=\frac{1}{\sigma}$, we have
\begin{eqnarray}\label{31}
\int_{\Omega}v^\sigma(t,\eta)\Theta_1(\eta)\,d\eta&=&\int_{\Omega}v^\sigma(t,\eta)\Theta_1^\sigma(\eta)\Theta_1^{1-\sigma}(\eta)\,d\eta\nonumber\\
&\geq&\left(\int_{\Omega}v(t,\eta)\Theta_1(\eta)\,d\eta\right)^{\sigma}\left(\int_{\Omega}\Theta_1(\eta)\,d\eta\right)^{1-\sigma}\nonumber\\
&=&\left(\int_{\Omega}v(t,\eta)\Theta_1(\eta)\,d\eta\right)^{\sigma}\nonumber\\
&=&J^\sigma(t),
\end{eqnarray}
 for all $t\geq0$, where we have used that $\displaystyle\int_{\Omega}\Theta_1(\eta)\,d\eta=1$. Inserting \eqref{31} into \eqref{33} we get
 \begin{equation}\label{34}
J(\tau)\geq J(0)+\frac{1}{2}\int_0^\tau J^\sigma(t)\,dt=:H(\tau).
\end{equation}
\noindent{\bf Step 2.}  Let $0<v_0\in  L^1(\mathbb{H}^n)\cap BC(\mathbb{H}^n)$, then $J(0)>0$. This implies, using \eqref{34}, that $J^\sigma(\tau)\geq H^\sigma(\tau)$ and $H^\sigma(\tau)\geq J^\sigma(0)>0$, so
$$H'(\tau)=\frac{1}{2}J^\sigma(\tau)\geq \frac{1}{2}H^\sigma(\tau),$$
i.e.
$$\frac{H'(\tau)}{H^\sigma(\tau)}\geq  \frac{1}{2},\quad\hbox{for all}\,\,\tau\geq0.$$
Integrating both sides over $(0,t)$, we arrive at
$$H(t)\geq\frac{1}{\left(J^{1-\sigma}(0)-(\sigma-1)\frac{t}{2}\right)^{\sigma-1}}.$$
Let
$$T^{*}=\frac{2J^{1-\sigma}(0)}{\sigma-1},$$ then we have
$$\sup_{\eta\in\mathbb{H}^n} v(t,\eta)\geq J(t)\geq H(t)\geq\frac{1}{\left(J^{1-\sigma}(0)-(\sigma-1)\frac{t}{2}\right)^{\sigma-1}}\longrightarrow\infty,\qquad\hbox{when}\,\,t\rightarrow T^{*}.$$
This completes the proof.
\end{proof}

\section{Degenerate parabolic equation}
In this section we consider the following degenerate parabolic equation
\begin{equation}\label{1}
\left\{
\begin{array}{ll}
\,\,\displaystyle{u_{t}=u^{q}\,\Delta_{\mathbb{H}}u+u^p,}&\displaystyle {t>0,\,\,\eta\in \mathbb{H}^n,}\\
\\
\displaystyle{u(0,\eta)= u_0(\eta)\geq0},&\displaystyle{\eta\in \mathbb{H}^n, }\\
\\
\displaystyle{u(t,\eta)\geq0},&\displaystyle {t>0,\,\,\eta\in \mathbb{H}^n,}
\end{array}
\right. \end{equation}
where $u_0\in L_{loc}^1(\mathbb{H}^n)$, $n\geq1$, $q\geq0$, $p>1$.\\

\subsection{Case of $0\leq q<1$}
We first consider the case $0\leq q<1$.
\begin{definition}\textup{(Weak solution of \eqref{1})}${}$\\
Let $u_0\in L^1_{loc}(\mathbb{H}^n)$ and $T>0$. We say that $u\geq0$ is a weak solution of \eqref{1} on $[0,T)\times\mathbb{H}^n$ if
$$u\in L_{loc}^p((0,T)\times\mathbb{H}^n)\cap L_{loc}^\infty((0,T);L_{loc}^1(\mathbb{H}^n)),\qquad u^{q}\,\Delta_{\mathbb{H}}u\in L_{loc}^1((0,T)\times\mathbb{H}^n),$$
and
\begin{eqnarray}\label{weaksolution1}
&{}&\int_{\mathbb{H}^n}u(\tau,\eta)\varphi(\tau,\eta)\,d\eta-\int_{\mathbb{H}^n}u(0,\eta)\varphi(0,\eta)\,d\eta\nonumber\\
&{}&\qquad\qquad=\int_0^\tau\int_{\mathbb{H}^n}u^p\varphi(t,\eta)\,d\eta\,dt+\int_0^\tau\int_{\mathbb{H}^n}u^{q}\,\Delta_{\mathbb{H}}u\,\varphi(t,\eta)\,d\eta\,dt\nonumber\\&{}&\qquad\qquad+\int_0^\tau\int_{\mathbb{H}^n}u\varphi_t(t,\eta)\,d\eta\,dt,
\end{eqnarray}
holds for all compactly supported $\varphi\in C^{1,0}_{t,x}([0,T)\times\mathbb{H}^n)$, and $0\leq\tau<T$. If $T=\infty$, we call $u$ a global  in time weak solution to \eqref{1}.
\end{definition}
We set
$$H^2(\mathbb{H}^n)=\{u\in L^2(\mathbb{H}^n);\,\,\nabla_{\mathbb{H}} u\in L^2(\mathbb{H}^n),\,\Delta_{\mathbb{H}} u\in L^2(\mathbb{H}^n)\}.$$
In order to get the nonexistence result of \eqref{1}, we need the following
\begin{lemma}\textup{(Weak solution of \eqref{1} $\Rightarrow$ Weak solution of \eqref{2})}\label{lemma1}${}$\\
Let $T>0$, $0\leq q<1$, $p>1$, and $0<u_0\in C(\mathbb{H}^n)\cap L^1_{loc}(\mathbb{H}^n)$. If $u>0$ is a positive weak solution of \eqref{1}  on $[0,T)\times\mathbb{H}^n$ such that $u\in C^{1,0}_{t,x}([0,T)\times\mathbb{H}^n)$ and $u(t,\cdotp)\in H^2(\mathbb{H}^n)$ for a.e. $t\in[0,T)$, then $v(t,\eta):=a u^{1-q}(t,\delta_b(\eta))$ is a positive weak solution of \eqref{2}  on $[0,T)\times\mathbb{H}^n,$ where
$$a=(1-q)^{\frac{1-q}{p-1}},\quad b=(1-q)^{\frac{p-1-q}{2(p-1)}},$$
$$\delta_b(\eta)=(bx,b y,b^2\tau),\quad\hbox{for all}\,\,\eta=(x,y,\tau)\in\mathbb{H}^n,$$
with
$$m=\frac{1}{1-q}\geq1,\quad\hbox{and}\quad \sigma=\frac{p-q}{1-q}>1.$$
\end{lemma}
\begin{proof}
Let $T>0$. Suppose that $u>0$ is a positive weak solution of \eqref{1}  on $[0,T)\times\mathbb{H}^n$ such that $u\in C^{1,0}_{t,x}([0,T)\times\mathbb{H}^n)$. Let $\psi\in C^{1,2}_{t,x}([0,T)\times\mathbb{H}^n)$ be a compactly supported test function. Let
$$\varphi(t,\eta)=u^{-q}(t,\eta)\psi(t,\delta_{\frac{1}{b}}(\eta)),$$
then $\varphi\in C^{1,0}_{t,x}([0,T)\times\mathbb{H}^n)$ and
\begin{eqnarray}\label{3}
&{}&\int_{\mathbb{H}^n}u^{1-q}(\tau,\eta)\psi(\tau,\delta_{\frac{1}{b}}(\eta))\,d\eta-\int_{\mathbb{H}^n}u^{1-q}(0,\eta)\psi(0,\delta_{\frac{1}{b}}(\eta))\,d\eta\nonumber\\
&{}&=\int_0^\tau\int_{\mathbb{H}^n}u^{p-q}(t,\eta)\psi(t,\delta_{\frac{1}{b}}(\eta))\,d\eta\,dt\nonumber\\
&{}&\quad+\int_0^\tau\int_{\mathbb{H}^n}\Delta_{\mathbb{H}}u(t,\eta)\,\psi(t,\delta_{\frac{1}{b}}(\eta))\,d\eta\,dt\nonumber\\
&{}&\quad+\int_0^\tau\int_{\mathbb{H}^n}u(t,\eta)\partial_t(u^{-q}(t,\eta)\psi(t,\delta_{\frac{1}{b}}(\eta)))\,d\eta\,dt,
\end{eqnarray}
for all $\tau\in[0,T)$. Using the integration by parts, we have
\begin{eqnarray}\label{4}
&{}&\int_0^{\tau}\int_{\mathbb{H}^n}u(t,\eta)\partial_t(u^{-q}(t,\eta)\psi(t,\delta_{\frac{1}{b}}(\eta)))\,d\eta\,dt\nonumber\\
&{}&=-q\int_0^{\tau}\int_{\mathbb{H}^n}u^{-q}(t,\eta)u_t(t,\eta)\psi(t,\delta_{\frac{1}{b}}(\eta))\,d\eta\,dt+\int_0^{\tau}\int_{\mathbb{H}^n}u^{1-q}(t,\eta)\psi_t(t,\delta_{\frac{1}{b}}(\eta))\,d\eta\,dt\nonumber\\
&{}&=-\frac{q}{1-q}\int_0^{\tau}\int_{\mathbb{H}^n}\partial_t(u^{1-q}(t,\eta))\psi(t,\delta_{\frac{1}{b}}(\eta))\,d\eta\,dt+\int_0^{\tau}\int_{\mathbb{H}^n}u^{1-q}(t,\eta)\psi_t(t,\delta_{\frac{1}{b}}(\eta))\,d\eta\,dt\nonumber\\
&{}&=\frac{q}{1-q}\int_0^{\tau}\int_{\mathbb{H}^n}u^{1-q}(t,\eta)\psi_t(t,\delta_{\frac{1}{b}}(\eta))\,d\eta\,dt+\int_0^{\tau}\int_{\mathbb{H}^n}u^{1-q}(t,\eta)\psi_t(t,\delta_{\frac{1}{b}}(\eta))\,d\eta\,dt\nonumber\\
&{}&\quad -\frac{q}{1-q}\int_{\mathbb{H}^n}u^{1-q}(\tau,\eta)\psi(\tau,\delta_{\frac{1}{b}}(\eta))\,d\eta+\frac{q}{1-q}\int_{\mathbb{H}^n}u_0^{1-q}(\eta)\psi(0,\delta_{\frac{1}{b}}(\eta))\,d\eta\nonumber\\
&{}&=\frac{1}{1-q}\int_0^{\tau}\int_{\mathbb{H}^n}u^{1-q}(t,\eta)\psi_t(t,\delta_{\frac{1}{b}}(\eta))\,d\eta\,dt-\frac{q}{1-q}\int_{\mathbb{H}^n}u^{1-q}(\tau,\eta)\psi(\tau,\delta_{\frac{1}{b}}(\eta))\,d\eta\nonumber\\
&{}&\quad+\frac{q}{1-q}\int_{\mathbb{H}^n}u_0^{1-q}(\eta)\psi(0,\delta_{\frac{1}{b}}(\eta))\,d\eta,
\end{eqnarray}
and
\begin{equation}\label{5}
\int_0^{\tau}\int_{\mathbb{H}^n}\Delta_{\mathbb{H}}u(t,\eta)\,\psi(t,\delta_{\frac{1}{b}}(\eta))\,d\eta\,dt=\int_0^{\tau}\int_{\mathbb{H}^n}u(t,\eta)\,\Delta_{\mathbb{H}}\left(\psi(t,\delta_{\frac{1}{b}}(\eta))\right)\,d\eta\,dt,
\end{equation}
for all $\tau\in[0,T)$. Inserting \eqref{4}-\eqref{5} into \eqref{3}, we obtain
\begin{eqnarray*}
&{}&\frac{1}{1-q}\int_{\mathbb{H}^n}u^{1-q}(\tau,\eta)\psi(\tau,\delta_{\frac{1}{b}}(\eta))\,d\eta-\frac{1}{1-q}\int_{\mathbb{H}^n}u_0^{1-q}(\eta)\psi(0,\delta_{\frac{1}{b}}(\eta))\,d\eta\\
&{}&=\int_0^\tau\int_{\mathbb{H}^n}u^{p-q}(t,\eta)\psi(t,\delta_{\frac{1}{b}}(\eta))\,d\eta\,dt+\int_0^\tau\int_{\mathbb{H}^n}u(t,\eta)\,\Delta_{\mathbb{H}}\left(\psi(t,\delta_{\frac{1}{b}}(\eta))\right)\,d\eta\,dt\\
&{}&\quad+\frac{1}{1-q}\int_0^\tau\int_{\mathbb{H}^n}u^{1-q}(t,\eta)\psi_t(t,\delta_{\frac{1}{b}}(\eta))\,d\eta\,dt,
\end{eqnarray*}
for all $\tau\in[0,T)$. Let $\widetilde{\eta}:=\delta_{\frac{1}{b}}(\eta)$ i.e. $\eta=\delta_b(\widetilde{\eta})$, then
\begin{eqnarray*}
&{}&\frac{1}{1-q}\int_{\mathbb{H}^n}u^{1-q}(\tau,\delta_b(\widetilde{\eta}))\psi(\tau,\widetilde{\eta})\,d\eta-\frac{1}{1-q}\int_{\mathbb{H}^n}u_0^{1-q}(\delta_b(\widetilde{\eta}))\psi(0,\widetilde{\eta})\,d\eta\\
&{}&=\int_0^\tau\int_{\mathbb{H}^n}u^{p-q}(t,\delta_b(\widetilde{\eta}))\psi(t,\widetilde{\eta})\,d\eta\,dt+b^{-2}\int_0^\tau\int_{\mathbb{H}^n}u(t,\delta_b(\widetilde{\eta}))\,\Delta_{\mathbb{H}}\psi(t,\widetilde{\eta})\,d\eta\,dt\\
&{}&\quad+\frac{1}{1-q}\int_0^\tau\int_{\mathbb{H}^n}u^{1-q}(t,\delta_b(\widetilde{\eta}))\psi_t(t,\widetilde{\eta})\,d\eta\,dt,
\end{eqnarray*}
for all $\tau\in[0,T)$. Using the fact that $d\eta=b^{Q}\,d\widetilde{\eta}$, and dividing the two sides by $b^{Q}$, we get
\begin{eqnarray}\label{6}
&{}&\frac{1}{1-q}\int_{\mathbb{H}^n}u^{1-q}(\tau,\delta_b(\widetilde{\eta}))\psi(\tau,\widetilde{\eta})\,d\widetilde{\eta}
-\frac{1}{1-q}\int_{\mathbb{H}^n}u_0^{1-q}(\delta_b(\widetilde{\eta}))\psi(0,\widetilde{\eta})\,d\widetilde{\eta}\nonumber\\
&{}&=\int_0^T\int_{\mathbb{H}^n}u^{p-q}(t,\delta_b(\widetilde{\eta}))\psi(t,\widetilde{\eta})\,d\widetilde{\eta}\,dt+b^{-2}\int_0^T\int_{\mathbb{H}^n}u(t,\delta_b(\widetilde{\eta}))\,\Delta_{\mathbb{H}}\psi(t,\widetilde{\eta})\,d\widetilde{\eta}\,dt\nonumber\\
&{}&\quad+\frac{1}{1-q}\int_0^T\int_{\mathbb{H}^n}u^{1-q}(t,\delta_b(\widetilde{\eta}))\psi_t(t,\widetilde{\eta})\,d\widetilde{\eta}\,dt,
\end{eqnarray}
for all $\tau\in[0,T)$. As $v(t,\widetilde{\eta})=a u^{1-q}(t,\delta_b(\widetilde{\eta}))$, we can easily obtain
\begin{equation}\label{7}
u^{p-q}(t,\delta_b(\widetilde{\eta})=(1-q)^{-\frac{p-q}{p-1}}\,v^{\sigma}(t,\widetilde{\eta}),\quad\,\,\frac{1}{1-q}u_0^{1-q}(\delta_b(\widetilde{\eta}))=(1-q)^{-\frac{p-q}{p-1}}\,v_0(\widetilde{\eta}),
\end{equation}
and
\begin{equation}\label{8}
b^{-2}u(t,\delta_b(\widetilde{\eta})=(1-q)^{-\frac{p-q}{p-1}}\,v^{m}(t,\widetilde{\eta}),\quad\,\,\frac{1}{1-q}u^{1-q}(t,\delta_b(\widetilde{\eta}))=(1-q)^{-\frac{p-q}{p-1}}\,v(t,\widetilde{\eta}).
\end{equation}
Putting \eqref{7}-\eqref{8} into \eqref{6}, and dividing the two sides by $(1-q)^{-\frac{p-q}{p-1}}$, we conclude that
\begin{eqnarray*}
&{}&\int_{\mathbb{H}^n}v_0(\widetilde{\eta})\psi(0,\widetilde{\eta})\,d\widetilde{\eta}-\int_{\mathbb{H}^n}v_0(\widetilde{\eta})\psi(0,\widetilde{\eta})\,d\widetilde{\eta}\\
&{}&\qquad=\int_0^T\int_{\mathbb{H}^n}v^{\sigma}(t,\widetilde{\eta})\psi(t,\widetilde{\eta})\,d\widetilde{\eta}\,dt+\int_0^T\int_{\mathbb{H}^n}v^{m}(t,\widetilde{\eta})\,\Delta_{\mathbb{H}}\psi(t,\widetilde{\eta})\,d\widetilde{\eta}\,dt\\ &{}&\qquad+\int_0^T\int_{\mathbb{H}^n}v(t,\widetilde{\eta})\psi_t(t,\widetilde{\eta})\,d\widetilde{\eta}\,dt,
\end{eqnarray*}
for all $\tau\in[0,T)$, i.e. $v$ is a weak solution of \eqref{2} on $[0,T)\times\mathbb{H}^n$.
\end{proof}
Set $BC(\mathbb{H}^n)= C(\mathbb{H}^n)\cap L^\infty(\mathbb{H}^n)$. Using Lemma \ref{lemma1} and Theorems \ref{theo1} and \ref{theo3} we conclude the following results.
\begin{theorem}
Let  $0<u_0\in BC(\mathbb{H}^n)\cap L^1(\mathbb{H}^n)$, $n\geq1$, $0\leq q<1$, $p>1$. If
$$q+1\leq p\leq p_c=q+1+\frac{2(1-q)}{Q},$$
then there are no positive global weak solutions $u\in C^{1,0}_{t,x}([0,\infty)\times\mathbb{H}^n)$ of \eqref{1} such that $u(t,\cdotp)\in H^2(\mathbb{H}^n)$ a.e. $t\in[0,\infty)$. Note that, in the case of $q+1<p$ we just need $u_0\in C(\mathbb{H}^n)\cap L^1(\mathbb{H}^n)$.
\end{theorem}
\begin{remark} When $q=0$, the critical exponent $p_c=1+\frac{2}{Q}$ coincides with the critical exponent obtained in \cite{Zhang} for the semilinear diffusion equations on $\mathbb{H}^n.$
\end{remark}
\begin{theorem}
Let  $0<u_0\in C(\mathbb{H}^n)\cap L^1(\mathbb{H}^n)$, $n\geq1$, $0\leq q<1$, $p>1$.  Assume that there exists a constant $\varepsilon_1>0$ such that, for every $0<\gamma<(1-q)Q$, the initial datum verifies the following assumption:
$$u_0(\eta)\geq \varepsilon_1(1+|\delta_{\frac{1}{b}}(\eta)|_{_{\mathbb{H}}}^2)^{-\frac{\gamma}{2(1-q)}}.$$
 If
$$q+1< p< q+1+\frac{2(1-q)}{\gamma},$$
then there are no positive global weak solutions $u\in C^{1,0}_{t,x}([0,\infty)\times\mathbb{H}^n)$ of \eqref{1} such that $u(t,\cdotp)\in H^2(\mathbb{H}^n)$ a.e. $t\in[0,\infty)$.
\end{theorem}
\begin{theorem}
Let  $n\geq1$, $0\leq q<1$, and $p>1$.\\
 If $q+1<p$, then for each $0<\widetilde{w}\in  L^1(\mathbb{H}^n)\cap BC(\mathbb{H}^n)$, there is $\widetilde{B}>0$ such that if $u_0=\widetilde{B}\widetilde{w}$ there are no positive global weak solutions $u\in C([0,\infty);L^1(\mathbb{H}^n))\cap C^{1,0}_{t,x}([0,\infty)\times\mathbb{H}^n)$ of \eqref{1} such that $u(t,\cdotp)\in H^2(\mathbb{H}^n)$ a.e. $t\in[0,\infty)$.\\
  If $q+1=p$, then for each $0< u_0 \in  L^1(\mathbb{H}^n)\cap BC(\mathbb{H}^n)$ there are no positive global weak solutions $u\in C([0,\infty);L^1(\mathbb{H}^n))\cap C^{1,0}_{t,x}([0,\infty)\times\mathbb{H}^n)$ of \eqref{1} such that $u(t,\cdotp)\in H^2(\mathbb{H}^n)$ a.e. $t\in[0,\infty)$.\\
More precisely, there exists a $T^{*}>0$ such that
$$\sup_{\eta\in\mathbb{H}^n} u(t,\eta)\longrightarrow\infty,\qquad\hbox{as}\,\,t\rightarrow T^{*}.$$
\end{theorem}

\subsection{The case of $q\geq1$} In this subsection, we present the results for the case $q\geq1$ and $1+q<p$.
\begin{theorem}\label{theo4}
Let  $n\geq1$, $q\geq1$, $p>1$. Suppose that $q+1<p$. For each $0<w\in C(\mathbb{H}^n)\cap L^\infty(\mathbb{H}^n)$, there is $A>0$ such that if $u_0=Aw$ then there are no positive global classical solutions of \eqref{1}.
\end{theorem}
\begin{remark} In Theorem \ref{theo4} there are no results for cases $q+1=p$ and $q+1>p$. Therefore, these questions are still open.
\end{remark}
\begin{proof}[Proof of Theorem \ref{theo4}] Suppose, on the contrary, that $u$ is a positive global classical solution of \eqref{1}, i.e. a positive classical solution of \eqref{1} on $[0,T]$ for all $T>0$. Let $\Omega\subset \mathbb{H}^n$ be a Heisenberg unit ball, and let $\lambda_1>0$ be the principal eigenvalue of $-\Delta_{ \mathbb{H}}$ with Dirichlet condition and $\Lambda>0$ its corresponding eigenfunction such that $\int_{\Omega}\Lambda(\eta)\,d\eta=1$ (The existence of such eigenvalue has been proved by Chen and Luo \cite{Chen1}). In order to get a contradiction, we are going to apply the energy method. We divide our proof into two cases.\\

\noindent{\bf \underline{Case of $q>1$.}}\\
\noindent{\bf Step 1.} Let
$$y(t):=\frac{1}{q-1}\int_{\Omega}u^{1-q}(t,\eta)\Lambda(\eta)\,d\eta,\quad t\in[0,T].$$  \\
As $u$ is a classical solution, we have
$$y\in C([0,T])\cap C^1((0,T]),$$
and
\begin{eqnarray*}
y^{\prime}(t)&=&-\int_{\Omega}\frac{u_t}{u^q}(t,\eta)\Theta(\eta)\,d\eta\\
&=&-\int_{\Omega}\Delta u(t,\eta)\Lambda(\eta)\,d\eta-\int_{\Omega}u^{p-q}(t,\eta)\Lambda(\eta)\,d\eta\\
&=&\lambda_1\int_{\Omega} u(t,\eta)\Lambda(\eta)\,d\eta+\int_{\partial\Omega} u(t,\sigma)\partial_{\nu}\Lambda(\sigma)\,d\sigma-\int_{\Omega}u^{p-q}(t,\eta)\Lambda(\eta)\,d\eta.
\end{eqnarray*}
It follows from the Hopf lemma on the Heisenberg group $\mathbb{H}^n$ (see \cite[Lemma 2.1]{BCutri}), that $\partial_{\nu}\Lambda\leq 0$ on $\partial\Omega$. Then we have
\begin{equation}\label{20}
y^{\prime}(t)\leq \lambda_1\int_{\Omega} u(t,\eta)\Lambda(\eta)\,d\eta-\int_{\Omega}u^{p-q}(t,\eta)\Lambda(\eta)\,d\eta.
\end{equation}
In order to apply the energy method, i.e. obtaining a differential inequality for $y(t)$, we need to estimate the right-hand side of \eqref{20}. Let $u_0=Aw$, where $0<w\in C(\mathbb{H}^n)\cap L^\infty(\mathbb{H}^n)$ and $A\gg1$ is a positive real number such that
$$A>(2\lambda_1)^{\frac{1}{p-q-1}}\left(\int_{\Omega}w^{1-q}(\eta)\Lambda(\eta)\,d\eta\right)^{\frac{1}{q-1}}.$$
This implies that $y_0:=y(0)<c_0$, with
$$c_0:=(q-1)^{-1}\left(2\lambda_1\right)^{-\frac{q-1}{p-q-1}}.$$
\noindent{\bf Step 2.}  We have $y(t)\leq c_0$, for all $t\in(0,T]$. Indeed, let $$T^{*}=\inf\{0<t\leq T;\,y(t)\leq c_0\}\leq T.$$ Since $y$ is continuous and $y(0)<c_0$, we have $T^{*}>0$. We claim that $T^*=T$. Otherwise, we have $y(t)<c_0$ for all $t\in(0,T^{*})$ and $y(T^{*})=c_0$, i.e. particularly, $y(t)\leq c_0$ for all $t\in[0,T^{*}]$. On the other hand, by H\"{o}lder's inequality for negative exponent
$$\int|fg|\,d\mu\geq\left(\int |f|^{r_1}\,d\mu\right)^{\frac{1}{r_1}}\left(\int |g|^{r_2}\,d\mu\right)^{\frac{1}{r_2}},\, \hbox{for all}\,\,r_1<0, 0<r_2<1, \,\frac{1}{r_1}+\frac{1}{r_2}=1,$$
 with $r_1=1-q$ and $r_2=\frac{q-1}{q}$, we have
\begin{eqnarray}\label{21}
\int_{\Omega}u(t,\eta)\Lambda(\eta)\,d\eta&=&\int_{\Omega}u(t,\eta)\Lambda^{-\frac{1}{q-1}}(\eta)\Lambda^{\frac{q}{q-1}}(\eta)\,d\eta\nonumber\\
&\geq&\left(\int_{\Omega}u^{1-q}(t,\eta)\Lambda(\eta)\,d\eta\right)^{-\frac{1}{q-1}}\left(\int_{\Omega}\Lambda(\eta)\,d\eta\right)^{\frac{q}{q-1}}\nonumber\\
&=&\left(\int_{\Omega}u^{1-q}(t,\eta)\Lambda(\eta)\,d\eta\right)^{-\frac{1}{q-1}}\nonumber\\
&=&(q-1)^{-\frac{1}{q-1}}y^{-\frac{1}{q-1}}(t),
\end{eqnarray}
 for all $t\in[0,T]$, where we have used that $\displaystyle\int_{\Omega}\Lambda(\eta)\,d\eta=1$. In addition, using the standard H\"{o}lder's inequality, we have
\begin{eqnarray*}
\int_{\Omega}u(t,\eta)\Lambda(\eta)\,d\eta&=&\int_{\Omega}u(t,\eta)\Lambda^{\frac{1}{p-q}}(\eta)\Lambda^{\frac{p-q-1}{p-q}}(\eta)\,d\eta\\
&\leq& \left(\int_{\Omega}u^{p-q}(t,\eta)\Lambda(\eta)\,d\eta\right)^{\frac{1}{p-q}}\left(\int_{\Omega}\Lambda(\eta)\,d\eta\right)^{\frac{p-q-1}{p-q}}\\
&=& \left(\int_{\Omega}u^{p-q}(t,\eta)\Lambda(\eta)\,d\eta\right)^{\frac{1}{p-q}},
\end{eqnarray*}
which implies, using \eqref{21} and $y(t)\leq c_0$, that
\begin{eqnarray}\label{22}
\int_{\Omega}u^{p-q}(t,\eta)\Lambda(\eta)\,d\eta&\geq&\left(\int_{\Omega}u(t,\eta)\Lambda(\eta)\,d\eta\right)^{p-q}\nonumber\\
&=&\left(\int_{\Omega}u(t,\eta)\Lambda(\eta)\,d\eta\right)^{p-q-1}\left(\int_{\Omega}u(t,\eta)\Lambda(\eta)\,d\eta\right)\nonumber\\
&\geq&(q-1)^{-\frac{p-q-1}{q-1}}y^{-\frac{p-q-1}{q-1}}(t)\left(\int_{\Omega}u(t,\eta)\Lambda(\eta)\,d\eta\right)\nonumber\\
&\geq&(q-1)^{-\frac{p-q-1}{q-1}}c_0^{-\frac{p-q-1}{q-1}}\left(\int_{\Omega}u(t,\eta)\Lambda(\eta)\,d\eta\right)\nonumber\\
&=&2\lambda_1\int_{\Omega}u(t,\eta)\Lambda(\eta)\,d\eta,
\end{eqnarray}
 for all $t\in(0,T^{*}]$. Therefore, by \eqref{20} and \eqref{22}, we arrive at
\begin{align*}y^{\prime}(t)&\leq\lambda_1\int_{\Omega}u(t,\eta)\Lambda(\eta)\,d\eta-2\lambda_1\int_{\Omega}u(t,\eta)\Lambda(\eta)\,d\eta\\&=-\lambda_1\int_{\Omega}u(t,\eta)\Lambda(\eta)\,d\eta,\end{align*}
which implies, using \eqref{21}, that
$$y^{\prime}(t)\leq -\lambda_1(q-1)^{-\frac{1}{q-1}}y^{-\frac{1}{q-1}}(t)\leq0,\qquad\hbox{ for all}\,\,t\in(0,T^{*}],
$$
and hence
$$c_0=y(T^{*})\leq y(0)=y_0<c_0;$$
contradiction.\\
\noindent{\bf Step 3.} From Step 2, we have $y(t)\leq c_0$, for all $t\in[0,T]$. This implies, using \eqref{21}-\eqref{22}, that
$$
y^{\prime}(t)\leq -\lambda_1(q-1)^{-\frac{1}{q-1}}y^{-\frac{1}{q-1}}(t),\qquad\hbox{ for all}\,\,t\in(0,T],
$$
so
$$0\leq y(t)\leq \left(y_0^{\frac{q}{q-1}}-c_1\,t\right)^{\frac{q-1}{q}},\qquad\hbox{ for all}\,\,t\in[0,T],$$
where $c_1=\lambda_1 q(q-1)^{-\frac{q}{q-1}}$, and particularly we have
$$T\leq c_1^{-1}y_0^{\frac{q}{q-1}}\leq c_1^{-1}c_0^{\frac{q}{q-1}},$$
which implies a contradiction by choosing from the beginning $T$ big enough, namely $T>c_1^{-1}c_0^{\frac{q}{q-1}}$. This completes the proof.\\

\noindent{\bf \underline{The case of $q=1$.}}\\
\noindent{\bf Step 1.} Let
$$y(t):=-\int_{\Omega}\ln (u(t,\eta))\Lambda(\eta)\,d\eta,\quad t\in[0,T].$$  \\
As $u$ is a classical solution, we have
$$y\in C([0,T])\cap C^1((0,T]),$$
and
\begin{eqnarray*}
y^{\prime}(t)&=&-\int_{\Omega}\frac{u_t}{u}(t,\eta)\Lambda(\eta)\,d\eta\\
&=&-\int_{\Omega}\Delta_{\mathbb{H}} u(t,\eta)\Lambda(\eta)\,d\eta-\int_{\Omega}u^{p-1}(t,\eta)\Lambda(\eta)\,d\eta\\
&=&\lambda_1\int_{\Omega} u(t,\eta)\Lambda(\eta)\,d\eta+\int_{\partial\Omega} u(t,\sigma)\partial_{\nu}\Lambda(\sigma)\,d\sigma\\
&{}&-\int_{\Omega}u^{p-1}(t,\eta)\Lambda(\eta)\,d\eta.
\end{eqnarray*}
As $\partial_{\nu}\Lambda\leq0$ on $\partial\Omega$ by the Hopf type lemma on the Heisenberg group $\mathbb{H}^n$ (see \cite[Lemma 2.1]{BCutri}), we arrive at
\begin{equation}\label{23}
y^{\prime}(t)\leq \lambda_1\int_{\Omega} u(t,\eta)\Lambda(\eta)\,d\eta-\int_{\Omega}u^{p-1}(t,\eta)\Lambda(\eta)\,d\eta.
\end{equation}
In order to apply the energy method, i.e. obtaining a differential inequality in $y(t)$, we need to estimate the right-hand side of \eqref{23}. Let $u_0=Aw$, where $0<w\in C(\mathbb{H}^n)\cap L^\infty(\mathbb{H}^n)$ and $A\gg1$ is a positive real number such that
$$A>(2\lambda_1)^{\frac{1}{p-2}}e^{-\int_{\Omega}\ln(w(\eta))\Lambda(\eta)\,d\eta}.$$
This implies that $y_0:=y(0)<c_2$, with
$$c_2:=-\frac{1}{p-2}\ln(2\lambda_1).$$
\noindent{\bf Step 2.}  We have $y(t)\leq c_2$, for all $t\in(0,T]$. Indeed, let $T^{*}=\inf\{0<t\leq T;\,y(t)\leq c_2\}\leq T$. Since $y$ is continuous and $y(0)<c_2$, we have $T^{*}>0$. We claim that $T^*=T$. Otherwise, we have $y(t)<c_2$ for all $t\in(0,T^{*})$ such that $y(T^{*})=c_2$, particularly we have $y(t)\leq c_2$ for all $t\in[0,T^{*}]$. On the other hand, by Jensen's inequality with $\displaystyle\int_{\Omega}\Lambda(\eta)\,d\eta=1$, we have
\begin{eqnarray}\label{24}
\int_{\Omega}u(t,\eta)\Lambda(\eta)\,d\eta&=&\int_{\Omega}e^{\ln u(t,\eta)}\Lambda(\eta)\,d\eta\nonumber\\
&\geq& e^{\int_{\Omega}\ln(u(t,\eta))\Lambda(\eta)\,d\eta}\nonumber\\
&=&e^{-y(t)},
\end{eqnarray}
 for all $t\in[0,T]$. In addition, using H\"{o}lder's inequality, we have
\begin{eqnarray*}
\int_{\Omega}u(t,\eta)\Lambda(\eta)\,d\eta&=&\int_{\Omega}u(t,\eta)\Lambda^{\frac{1}{p-1}}(\eta)\Lambda^{\frac{p-2}{p-1}}(\eta)\,d\eta\\
&\leq& \left(\int_{\Omega}u^{p-1}(t,\eta)\Lambda(\eta)\,d\eta\right)^{\frac{1}{p-1}}\left(\int_{\Omega}\Lambda(\eta)\,d\eta\right)^{\frac{p-2}{p-1}}\\
&=& \left(\int_{\Omega}u^{p-1}(t,\eta)\Lambda(\eta)\,d\eta\right)^{\frac{1}{p-1}},
\end{eqnarray*}
which implies, using \eqref{24} and $y(t)\leq c_2$, that
\begin{eqnarray}\label{25}
\int_{\Omega}u^{p-1}(t,\eta)\Lambda(\eta)\,d\eta&\geq&\left(\int_{\Omega}u(t,\eta)\Lambda(\eta)\,d\eta\right)^{p-1}\nonumber\\
&=&\left(\int_{\Omega}u(t,\eta)\Lambda(\eta)\,d\eta\right)^{p-2}\left(\int_{\Omega}u(t,\eta)\Lambda(\eta)\,d\eta\right)\nonumber\\
&\geq&e^{-(p-2)y(t)}\left(\int_{\Omega}u(t,\eta)\Lambda(\eta)\,d\eta\right)\nonumber\\
&\geq&e^{-(p-2)c_2}\left(\int_{\Omega}u(t,\eta)\Lambda(\eta)\,d\eta\right)\nonumber\\
&=&2\lambda_1\int_{\Omega}u(t,\eta)\Lambda(\eta)\,d\eta,
\end{eqnarray}
for all $t\in(0,T^{*}]$. Therefore, by \eqref{23} and \eqref{25}, we get
\begin{align*}
y^{\prime}(t)&\leq \lambda_1\int_{\Omega}u(t,\eta)\Lambda(\eta)\,d\eta-2\lambda_1\int_{\Omega}u(t,\eta)\Lambda(\eta)\,d\eta\\&=-\lambda_1\int_{\Omega}u(t,\eta)\Lambda(\eta)\,d\eta,\end{align*}
and then, by using \eqref{24}, we we arrive at
$$y^{\prime}(t)\leq-\lambda_1\,e^{-y(t)}\leq0,\qquad\hbox{ for all}\,\,t\in(0,T^{*}],
$$
and hence
$$c_2=y(T^{*})\leq y(0)=y_0<c_2;$$
contradiction.\\
\noindent{\bf Step 3.} From Step 2, we have $y(t)\leq c_2$, for all $t\in[0,T]$. This implies, using \eqref{24}-\eqref{25}, that
$$
y^{\prime}(t)\leq -\lambda_1\,e^{-y(t)},\qquad\hbox{ for all}\,\,t\in(0,T],
$$
so
$$0<e^{y(t)}\leq e^{y_0}-\lambda_1t\qquad\hbox{ for all}\,\,t\in[0,T],$$
 and hence
$$T< \frac{e^{y_0}}{\lambda_1}\leq \frac{e^{c_2}}{\lambda_1},$$
which implies a contradiction by choosing from the beginning $T$ big enough, namely $T\geq\frac{e^{c_2}}{\lambda_1}$. This completes the proof.
\end{proof}

\section*{Declaration of competing interest}
	The authors declare that there is no conflict of interest.

\bibliographystyle{amsplain}

\end{document}